\documentclass[a4paper,12pt,reqno]{amsart}
\usepackage{amsmath, amssymb,amsthm}
\usepackage{enumerate}
\usepackage{cite}

\nonstopmode \numberwithin{equation}{section}
\newtheorem{theorem}{Theorem}[section]

\newtheorem{corollary}{Corollary}[section]

\newtheorem{remark}{Remark}[section]
\newtheorem{problem}{Problem}

\begin{document}
\bibliographystyle{amsplain}

\title{{Convexity of the Generalized Integral Transform and Duality Techniques }}

\author{
Satwanti Devi
}
\address{
Department of  Mathematics  \\
Indian Institute of Technology, Roorkee-247 667,
Uttarakhand,  India
}
\email{ssatwanti@gmail.com}

\author{
A. Swaminathan
}
\address{
Department of  Mathematics  \\
Indian Institute of Technology, Roorkee-247 667,
Uttarakhand,  India
}
\email{swamifma@iitr.ernet.in, mathswami@gmail.com}

\bigskip

\begin{abstract}
Let $\mathcal{W}_{\beta}^\delta(\alpha,\gamma)$ be the class of
normalized analytic functions $f$ defined in the domain $|z|<1$
satisfying
\begin{align*}
{\rm Re\,} e^{i\phi}\left(\dfrac{}{}(1\!-\!\alpha\!+\!2\gamma)\!\left({f}/{z}\right)^\delta
+\left(\alpha\!-\!3\gamma+\gamma\left[\dfrac{}{}\left(1-{1}/{\delta}\right)\left({zf'}/{f}\right)+
{1}/{\delta}\left(1+{zf''}/{f'}\right)\right]\right)\right.\\
\left.\dfrac{}{}\left({f}/{z}\right)^\delta \!\left({zf'}/{f}\right)-\beta\right)>0,
\end{align*}
with the conditions $\alpha\geq 0$, $\beta<1$, $\gamma\geq 0$,
$\delta>0$ and $\phi\in\mathbb{R}$. Moreover, for
$0<\delta\leq\frac{1}{(1-\zeta)}$, $0\leq\zeta<1$, the class
$\mathcal{C}_\delta(\zeta)$ be the subclass of normalized
analytic functions such that
\begin{align*}
{\rm Re}{\,}\left(1/\delta\left(1+zf''/f'\right)+(1-1/\delta)\left({zf'}/{f}\right)\right)>\zeta,\quad |z|<1.
\end{align*}
In the present work, the sufficient conditions on $\lambda(t)$
are investigated, so that the generalized integral transform
\begin{align*}
V_{\lambda}^\delta(f)(z)= \left(\int_0^1 \lambda(t) \left({f(tz)}/{t}\right)^\delta dt\right)^{1/\delta},\quad |z|<1,
\end{align*}
carries the functions from
$\mathcal{W}_{\beta}^\delta(\alpha,\gamma)$ into
$\mathcal{C}_\delta(\zeta)$. Several interesting applications
are provided for special choices of $\lambda(t)$.
\end{abstract}




\maketitle

\pagestyle{myheadings} \markboth{Satwanti Devi and A. Swaminathan }{ Convexity of Generalized
Integral Transform and Duality Technique}

\section{introduction}
Let $\mathcal{A}$ be the class of all normalized analytic functions $f$ defined in the region $\mathbb{D}=\{
z\in\mathbb{C}: |z|<1\}$ with the condition $f(0)=f'(0)-1=0$ and $\mathcal{S}\subset\mathcal{A}$ be the class of all
univalent functions in $\mathbb{D}$.
We are interested in the following problem.

\begin{problem}\label{Prob1:Genl_operator}
Given  $\lambda(t):[0,1]\!\rightarrow\! \mathbb{R}$ be a non-negative
integrable function with the condition $\int_0^1\lambda(t)
dt=1$, then for $f$ in a particular class of analytic functions, the generalized integral
transform defined by
\begin{align}\label{eq-weighted-integralOperator}
V_{\lambda}^\delta(f)(z):=\left(\int_0^1 \lambda(t) \left(\dfrac{f(tz)}{t}\right)^\delta dt\right)^{1/\delta},
\quad \delta>0\quad {\rm and}\quad z\in\mathbb{D}
\end{align}
is in one of the subclasses of ${\mathcal{S}}$.
\end{problem}

This problem, for the case $\delta=1$ was first stated by R. Fournier and S. Ruscheweyh \cite{FourRusExtremal} by examining the
characterization of two extremal problems. They
considered the functions $f$ in the class ${\mathcal{P}}_{\beta}$, where
\begin{align*}
{\mathcal{P}}_{\beta} = \left\{ \dfrac{}{} f\in {\mathcal{A}}: {\rm Re\,} \left(\dfrac{}{}e^{i\alpha}(f'(z)-\beta)\right)>0, \quad
\alpha\in {\mathbb{R}}, \quad z\in{\mathbb{D}}\right\}.
\end{align*}
such that the integral operator $V_{\lambda}(f)(z): V_{\lambda}^1(f)(z)$ is in the class ${\mathcal{S}^{\ast}}$ of functions that map
${\mathbb{D}}$ onto domain that are starlike with respect to origin using duality techniques. Same problem was solved by
R.M. Ali and V.Singh \cite{Ali} for functions $f$ in the class ${\mathcal{P}}_{\beta}$ so that the
integral operator $V_{\lambda}(f)(z)$ is in the class ${\mathcal{C}}$ of functions that map ${\mathbb{D}}$ onto domain that are convex.

The integral operator $V_{\lambda}(f)(z)$ contains some of the
well-known operator such as Bernardi, Komatu and Hohlov as its
special cases for particular choices of $\lambda(t)$, which has
been extensively studied by various authors (for details see
\cite{Ali, saigo, DeviPascu} and references therein).
Generalization of the class ${\mathcal{P}}_{\beta}$ for studying the above problem with reference to the operator $V_{\lambda}(f)(z)$  were
considered by several researchers in the recent past and interesting applications were obtained. For
most general result in this direction, see \cite{DeviPascuOrder} and references therein.

The integral operator \eqref{eq-weighted-integralOperator} and
its generalization was considered in the work of I. E.
Bazilevi\v c \cite{BazilevicIntegOper} (for more details see
\cite{Aghalary, SinghIntOperator}). Problem $\ref{Prob1:Genl_operator}$ for the generalized integral operator
$V_{\lambda}^\delta(f)(z)$ relating starlikeness was investigated by A. Ebadian et al. in \cite{Aghalary} by considering
the class
{\small{
\begin{align*}
P_\alpha(\delta,\beta):=
\left\{f\in\mathcal{A}\,,\,\exists\,\phi\in\mathbb{R}:\,{\rm Re\,} e^{i\phi}
\left((1-\alpha)\left(\dfrac{f}{z}\right)^\delta
+\alpha\left(\dfrac{f}{z}\right)^\delta \left(\dfrac{zf'}{f}\right)-\beta\right)>0,\,z\in\mathbb{D}\right\}
\end{align*}}}
with $\alpha \geq0$, $\beta<1$ and $\delta>0$.
The authors of the present work have generalized the starlikeness criteria \cite{DeviGenlStar}
by considering the following subclass of ${\mathcal{S}}^{\ast}$
\begin{align}\label{eq-gener:starlike-related:classes}
f\in \mathcal{S}^\ast_s(\zeta)\,\,\Longleftrightarrow\,\,
z^{1-\delta}f^{\delta}\in\mathcal{S}^\ast(\xi),
\end{align}
for $\xi=1-\delta+\delta\zeta$ and $0\leq\xi<1$ where
$\mathcal{S}^{\ast}(\xi)$ is the class having the analytic characterization
\begin{align*}
{\rm{Re}}{\,} \left(\dfrac{zf^\prime}{f}\right)>\xi,\quad 0\leq\xi<1, \quad z\in\mathbb{D}.
\end{align*}
Note that $\mathcal{S}^{\ast}: =\mathcal{S}^{\ast}(0)$.
In \cite{DeviGenlStar}, this problem was investigated
by considering the integral operator acting on the most generalized class of ${\mathcal{P}}_{\beta}$, related to the present context,
which is defined as follows.
{\small{
\begin{align*}
\mathcal{W}_{\beta}^\delta(\alpha,\gamma)\!:\!=\!\left\{\!f\!\in\!\mathcal{A}:
{\rm Re\,} e^{i\phi}\left((1\!-\!\alpha\!+\!2\gamma)\!\left(\frac{f}{z}\right)^\delta
\!+\!\left(\alpha\!-\!3\gamma\!+\!\gamma\left[\left(1\!-\!\frac{1}{\delta}\right)\left(\frac{zf'}{f}\right)\!+\!
\frac{1}{\delta}\left(1\!+\!\frac{zf''}{f'}\right)\right]\right)\right.\right.\\
\left.\left.\left(\frac{f}{z}\right)^\delta \!\left(\frac{zf'}{f}\right)-\beta\right)>0,
\,\,z\in\mathbb{D},\,\,\phi\in\mathbb{R}\right\}.
\end{align*}}}

\noindent Here, $\alpha \geq 0$, $\beta<1$, $\gamma\geq 0$ and $\phi\in {\mathbb{R}}$. Note that
$\mathcal{W}_{\beta}^\delta(\alpha,0)\equiv P_\alpha(\delta,\beta)$ is the class considered by A. Ebadian et al in \cite{Aghalary},
$R_\alpha(\delta,\beta):\equiv \mathcal{W}_{\beta}^\delta(\alpha+\delta+\delta\alpha,\delta\alpha)$ is a closely related class and
$\mathcal{W}_{\beta}^1(\alpha,\gamma)\equiv\mathcal{W}_{\beta}(\alpha,\gamma)$ introduced by R.M. Ali et al in \cite{AbeerS}.

As the investigation of this generalization provided fruitful results, we are interested in considering further geometric properties
of the generalized integral operator given by \eqref{eq-weighted-integralOperator} for $f\in \mathcal{W}_{\beta}^\delta(\alpha,\gamma)$.
Motivated, by the well-known Alexander theorem \cite{DU},
\begin{align*}
f\in\mathcal{C}(\xi)\Longleftrightarrow zf '\in\mathcal{S}^\ast(\xi),
\end{align*}
where ${\mathcal{C}}(0)={\mathcal{C}}$,
we consider the subclass
\begin{align}\label{iff:relation:convex+starlike}
f\in \mathcal{C}_\delta(\zeta)\Longleftrightarrow
(z^{2-\delta}f^{\delta-1}f')\in\mathcal{S}^\ast(\xi),
\end{align}
where $\xi:=1-\delta+\delta\zeta$ with the conditions
$1-\frac{1}{\delta}\leq\zeta<1$, $0\leq\xi<1$ and
$\delta\geq1$. In the sequel, the term $\xi$ is used to denote
$(1-\delta+\delta\zeta)$. From the above expression, it is easy
to observe that the class $\mathcal{C}_\delta(\zeta)$ and
$\mathcal{C}(\xi)$ are equal, when $\delta=1$.

The class $\mathcal{C}_\delta(\zeta)$ given in \eqref{iff:relation:convex+starlike}
is related to the class of $\alpha$- convex of order $\zeta$ $(0\leq\zeta<1)$ that were
introduced in the work of P. T. Mocanu \cite{MocanuAlphaConvex} and
defined analytically as
\begin{align*}
{\rm Re}{\,}\left(\left(1-\alpha\right)\left(\dfrac{zf'(z)}{f(z)}\right)+\alpha\left(1+\dfrac{zf''(z)}{f'(z)}\right)\right)>\zeta,
\quad (1-\zeta)\leq\alpha<\infty.
\end{align*}
Clearly the class $\mathcal{C}_\delta(\zeta)$ is nothing but the
subclass of $\mathcal{S}$ consisting of ${1}/{\delta}\,$-
convex functions of order $\zeta$.

Having provided all the required information from the literature, in what follows, we obtain sharp estimates
for the parameter $\beta$ so that the generalized
integral operator \eqref{eq-weighted-integralOperator}
maps the function from $\mathcal{W}_\beta^\delta(\alpha,\gamma)$ into
$\mathcal{C}_\delta(\zeta)$, where
$0<\delta\leq\frac{1}{(1-\zeta)}$ and $0\leq\zeta<1$.
Duality techniques, given in \cite{Rus} provide the platform for the entire study of this manuscript. One of the particular tool in this regard
is the convolution or Hadamard product of two
functions
\begin{align*}
f_1(z)=\sum_{n=0}^\infty a_n z^n\quad {\rm and}\quad f_2(z)=\sum_{n=0}^\infty b_n z^n,\quad z\in\mathbb{D},
\end{align*}
given by
\begin{align*}
(f_1\ast f_2)(z)\!=\!\displaystyle\sum_{n=0}^{\infty}a_nb_n
z^n.
\end{align*}

Furthermore, consider the complex parameters $c_i$ $(i=0,1,
\ldots,p)$ and $d_j$ $(j=0,1,\ldots,q)$ with
$d_j\neq0,-1,\ldots$ and $p\leq q+1$. Then, in the region
$\mathbb{D}$, the generalized hypergeometric function is given
by
\begin{align*}
{\,}_{p}F_q\left(\!\!\!\!\!
\begin{array}{cll}&\displaystyle c_1,\ldots,c_p
\\
&\displaystyle d_1,\ldots,d_q
\end{array};z\right)
=\sum_{n=0}^\infty
\dfrac{(c_1)_n\ldots,(c_p)_n}{(d_1)_n\ldots,(d_q)_nn!}z^n,\quad z\in\mathbb{D},
\end{align*}
that can also be represented as
$_{p}F_{q}(c_1,\ldots,c_p;d_1,\ldots,d_q;z)$ or $_{p}F_{q}$. In
particular, $_{2}F_{1}$ is the well-known Gaussian
hypergeometric function. For any natural number $n$, the
Pochhammer symbol or shifted factorial $(\varepsilon)_n$ is
defined as $(\varepsilon)_0=1$ and
$(\varepsilon)_n=\varepsilon(\varepsilon+1)_{n-1}$.

The paper is organized as follows: Necessary and sufficient conditions are
obtained in Section \ref{Sec-Gener:Convex-Main:Result} that
ensures
$V_\lambda^\delta(\mathcal{W}_\beta^\delta(\alpha,\gamma))\!\subset\!
\mathcal{C}_\delta(\zeta)$. The simpler sufficient criterion
are derived in Section \ref{Sec-Gener:Convex-Suff:Cond}, which
are further implemented to find many interesting applications  involving various integral operators
for special choices of $\lambda(t)$.
\section{preliminaries}
The parameters $\mu,\,\nu\geq0$ introduced in \cite{AbeerS} are
used for further analysis that are defined by the relations

\begin{align}\label{eq-mu+nu}
\mu\nu=\gamma\quad  \text{and}\quad \mu+\nu=\alpha-\gamma.
\end{align}
Clearly $\eqref{eq-mu+nu}$ leads to two cases.
\begin{itemize}
\item[{\rm{(i)}}] $\gamma=0 \, \Longrightarrow \mu=0,\,
    \nu=\alpha \geq 0$.
\item[{\rm{(ii)}}] $\gamma>0 \, \Longrightarrow \mu>0,\,
    \nu>0$.
\end{itemize}
Define the auxiliary function

\begin{align}\label{eq-Gener:convex-psi:munu}
\psi_{\mu,\nu}^\delta(z):=\sum_{n=0}^\infty\dfrac{\delta^2}{(\delta+n\mu)(\delta+n\nu)}z^n
=\int_0^1\int_0^1\dfrac{1}{(1-u^{\nu/\delta/}v^{\mu/\delta}z)}dudv,
\end{align}
which by a simple computation gives

\begin{align}\label{eq-Gener:convex-Phi:munu}
\Phi_{\mu,\nu}^\delta(z):=\left(z\psi_{\mu,\nu}^\delta(z)\right)'=\sum_{n=0}^\infty\dfrac{(n+1)\delta^2}{(\delta+n\nu)(\delta+n\mu)}z^n
\end{align}
\begin{align}\label{eq-Gener:convex-Upsilon:munu}
{\rm and}\quad\quad\Upsilon_{\mu,\nu}^\delta(z):=\left(z\left(z\psi_{\mu,\nu}^\delta(z)\right)'\right)'=
\sum_{n=0}^\infty\dfrac{(n+1)^2\delta^2}{(\delta+n\nu)(\delta+n\mu)}z^n.
\end{align}
Taking $\gamma=0$ $(\mu=0,\,\nu=\alpha\geq0)$, let
$q_{0,\alpha}^\delta(t)$ be the solution of the differential
equation

\begin{align}\label{eq-generalized:convex-q:gamma0}
\dfrac{d}{dt}\left(t^{\delta/\alpha}q_{0,\alpha}^\delta(t)\right)=\dfrac{\delta t^{\delta/\alpha-1}}{\alpha}
\left(\left(1-\dfrac{1}{\delta}\right)\dfrac{(1-\xi\left(1+t\right))}{(1-\xi)\left(1+t\right)^2}
+\left(\dfrac{1}{\delta}\right)\dfrac{(1-t-\xi\left(1+t\right))}{(1-\xi)\left(1+t\right)^3}\right),
\end{align}
with the initial condition $q_\alpha^\delta(0)=1$. Then the
solution of \eqref{eq-generalized:convex-q:gamma0} is given by

\begin{align*}
q_{0,\alpha}^\delta(t)
=\dfrac{\delta t^{-\delta/\alpha}}{\alpha}
\int_0^t\left(\left(1-\dfrac{1}{\delta}\right)\dfrac{(1-\xi\left(1+s\right))}{(1-\xi)\left(1+s\right)^2}
+\left(\dfrac{1}{\delta}\right)\dfrac{(1-s-\xi\left(1+s\right))}{(1-\xi)\left(1+s\right)^3}\right)s^{\delta/\alpha-1}ds.
\end{align*}
Also, for the case $\gamma>0$ $(\mu>0,\,\nu>0)$, let
$q_{\mu,\nu}^\delta(t)$ be the solution of the differential
equation

{\small{
\begin{align}\label{eq-Gener:Convex-q:gamma>0}
\dfrac{d}{dt}\left(t^{\delta/\nu}q_{\mu,\nu}^\delta(t)\right)=\dfrac{\delta^2 t^{\delta/\nu-1}}{\mu\nu}
\int_0^1\!\!\left(\!\left(1\!-\!\dfrac{1}{\delta}\right)\dfrac{(1\!-\!\xi(1\!+\!st))}{(1\!-\!\xi)(1\!+\!st)^2}\!+\!
\left(\dfrac{1}{\delta}\right)\dfrac{(1\!-\!st\!-\!\xi(1\!+\!st))}{(1\!-\!\xi)(1\!+\!st)^3}\right)s^{\delta/\mu-1}ds,
\end{align}}}
with the initial condition $q_{\mu,\nu}^\delta(0)=1$. Then the
solution of \eqref{eq-Gener:Convex-q:gamma>0} is given as

{\small{
\begin{align}\label{eq-Gener:Convex-Integral:q:gamma>0}
q_{\mu,\nu}^\delta(t)\!=\!\dfrac{\delta^2}{\mu\nu}\!\int_0^1\!\!\!\!\int_0^1\!\left(\!\left(1\!-\!\dfrac{1}{\delta}\right)
\dfrac{(1\!-\!\xi\left(1\!+\!trs\right))}{(1\!-\!\xi)\left(1\!+\!trs\right)^2}\!+\!
\left(\dfrac{1}{\delta}\right)\dfrac{(1\!-\!trs\!-\!\xi\left(1\!+\!trs\right))}
{(1\!-\!\xi)\left(1+trs\right)^3}\right)r^{\delta/\nu-1}s^{\delta/\mu-1}drds,
\end{align}}}

Furthermore, for given $\lambda(t)$ and $\delta>0$, we
introduce the functions

\begin{align}\label{eq-weighted:Lambda}
\Lambda_\nu^\delta(t):=\displaystyle\int_t^1\dfrac{\lambda(s)}{s^{{\delta}/{\nu}}}ds,\quad \nu>0,
\end{align}
and

\begin{align}\label{eq-weighted:Pi}
\Pi_{\mu,\nu}^\delta(t):=
\left\{\!
\begin{array}{cll}\displaystyle \int_t^1\!\!\dfrac{\Lambda_\nu^\delta(s)}{s^{{\delta}/{\mu}-{\delta}/{\nu}+1}}ds
&\quad\gamma\!>\!0 \,(\mu\!>\!0,\nu\!>\!0),\\\\
\Lambda_\alpha^\delta(t)
&\quad\gamma\!=\!0\,(\mu\!=\!0,\nu\!=\!\alpha\!\geq\!0)
\end{array}\right.
\end{align}
which are positive on $t\in(0,1)$ and integrable on
$t\in[0,1]$.

For $\delta=1$, these information coincide with the one given
in \cite{MahnazC}. \eqref{eq-weighted:Lambda} and \eqref{eq-weighted:Pi} are also considered in \cite{DeviGenlStar} and \cite{DeviGenlUniv}.
In \cite{DeviGenlStar}, the investigations are related to
$V_\lambda^\delta(f)(z)\in\mathcal{S}^\ast_s(\zeta)$, whenever $f\in\mathcal{W}_\beta^\delta(\alpha,\gamma)$, whereas
various other inclusion properties, in particular,
$V_\lambda^\delta(f)(z)\in\mathcal{W}_{\beta_1}^{\delta_1}(\alpha_1,\gamma_1)$, whenever $f\in\mathcal{W}_{\beta_2}^{\delta_2}(\alpha_2,\gamma_2)$
are investigated in \cite{DeviGenlUniv}.
\section{main results}\label{Sec-Gener:Convex-Main:Result}
The following result establishes both the necessary and
sufficient conditions that ensure
$F_\delta(z):=V_{\lambda}^\delta(f)(z)\in\mathcal{C}_\delta(\zeta)$,
whenever $f\in\mathcal{W}_{\beta}^\delta(\alpha,\gamma)$.
\begin{theorem}\label{Thm-Gener:Convex-Mainresult}
Let $\mu\!\geq\!0$, $\nu\!\geq\!0$ are given by the relation in
\eqref{eq-mu+nu} and
$\left(1-\frac{1}{\delta}\right)\leq\zeta\leq
\left(1-\frac{1}{2\delta}\right)$ where $\delta\geq1$. Let
$\beta\!<\!1$ satisfy the condition

\begin{align}\label{Beta-Cond-Generalized:Convex}
\dfrac{\beta-\frac{1}{2}}{1-\beta}=-\int_0^1 \lambda(t) q_{\mu,\nu}^\delta(t) dt,
\end{align}
where $q_{\mu,\nu}^\delta(t)$ is defined by the differential
equation \eqref{eq-generalized:convex-q:gamma0} for $\gamma=0$
and \eqref{eq-Gener:Convex-q:gamma>0} for $\gamma>0$. Further
assume that the functions given in \eqref{eq-weighted:Lambda}
and \eqref{eq-weighted:Pi} attains

\begin{align*}
\lim_{t\rightarrow 0^+} t^{{\delta}/{\nu}}\Lambda_\nu^\delta(t)\rightarrow 0\quad{\rm and}\quad
\lim_{t\rightarrow 0^+}t^{{\delta}/{\mu}}\Pi_{\mu,\nu}^\delta(t)\rightarrow 0.
\end{align*}
Then for $f(z)\in\mathcal{W}_{\beta}^\delta(\alpha,\gamma)$,
the function $F_\delta:=V_\lambda^\delta(f(z))\!\in\!
\mathcal{C}_\delta(\zeta)$ iff,
$\mathcal{M}_{\Pi_{\mu,\nu}^\delta}(h_\xi)(z)\geq0$, where

\begin{align*}
\mathcal{M}_{\Pi_{\mu,\nu}^\delta}(h_\xi)(z)\!:=\!\left\{\!\!
\begin{array}{cll}\displaystyle \int_0^1 t^{{\delta}/{\mu}-1}\Pi_{\mu,\nu}^\delta(t){\,} h_{\xi,\delta,z}(t)dt,\quad
&\gamma>0{\,}(\mu>0,{\,}\nu>0),\\\\
\displaystyle\int_0^1 t^{{\delta}/{\alpha}-1}\Lambda_\alpha^\delta(t){\,} h_{\xi,\delta,z}(t)dt,\quad
&\gamma=0{\,}(\mu=0,{\,}\nu=\alpha\geq0),
\end{array}\right.
\end{align*}
and

\begin{align*}
h_{\xi,\delta,z}(t):=\left(1\!-\!\dfrac{1}{\delta}\right)\!
\!\left({\rm Re}{\,}\dfrac{h_\xi(tz)}{tz}\!-\!\dfrac{1\!-\!\xi(1\!+\!t)}{(1\!-\!\xi)(1\!+\!t)^2}\right)\!+\!\left(\dfrac{1}{\delta}\right)\!
\left({\rm Re}{\,}h_\xi'(tz)\!-\!\dfrac{1\!-\!t\!-\!\xi(1\!+\!t)}{(1\!-\!\xi)(1\!+\!t)^3}\right)
\end{align*}
for the function

\begin{align}\label{eq-h(z)-extremal-starlike}
h_\xi(z):=z\left(\dfrac{1+\frac{\epsilon+2\xi-1}{2(1-\xi)}z}{(1-z)^2}\right),\quad |\epsilon|=1
\end{align}
and $\xi:=1-\delta(1-\zeta)$, $0\leq\xi\leq1/2$. The value of
$\beta$ is sharp.
\end{theorem}
\begin{proof}
From \eqref{iff:relation:convex+starlike}, it is clear that

\begin{align}\label{iff:relation:convex+starlike-integral}
F_\delta\in\mathcal{C}_\delta(\zeta)\Longleftrightarrow\left(z^{2-\delta}(F_\delta)^{\delta-1}F_\delta'\right)\in\mathcal{S}^\ast(\xi)
\end{align}
where $\xi:=1-\delta(1-\zeta)$. Thus, to prove
$\mathcal{M}_{\Pi_{\mu,\nu}^\delta}(h_\xi)\geq 0$ using the
given hypothesis, it is required to show that the function
$z^{2-\delta}(F_\delta)^{\delta-1}F_\delta'$ is univalent and
satisfy the order of starlikeness condition, and conversely.
Let

{\small{
\begin{align}\label{eq-generalized:starlike-H}
H(z):=(1\!-\!\alpha\!+\!2\gamma)\!\left(\frac{f}{z}\right)^\delta
\!+\!\left(\alpha\!-\!3\gamma+\gamma\left[\left(1\!-\!\frac{1}{\delta}\right)\left(\frac{zf'}{f}\right)\!+\!
\frac{1}{\delta}\left(1\!+\!\frac{zf''}{f'}\right)\right]\right)
\left(\frac{f}{z}\right)^\delta\!\left(\frac{zf'}{f}\right).
\end{align}}}
Using the relation \eqref{eq-mu+nu} in
\eqref{eq-generalized:starlike-H} gives

\begin{align*}
H(z)
=\dfrac{\mu\nu}{\delta^2}z^{1-\delta/\mu}\left(z^{\delta/\mu-\delta/\nu+1} \left(
z^{\delta/\nu}\left(\dfrac{f}{z}\right)^\delta\right)'\right)'.
\end{align*}
Further, set $G(z)=(H(z)-\beta)/(1-\beta)$, then there exist
some $\phi\in\mathbb{R}$, such that ${\rm
Re}\left(e^{i\phi}{\,}G(z)\right)>0$. Hence by the duality
principle \cite[p. 22]{Rus}, we may confine to the function
$f(z)$ for which $G(z)={(1+xz)}/{(1+yz)}$, where $|x|=|y|=1$,
which directly implies

\begin{align*}
\dfrac{\mu\nu}{\delta^2}{\,} z^{1-\delta/\mu}\left(z^{\delta/\mu-\delta/\nu+1}
\left( z^{\delta/\nu}\left(\dfrac{f}{z}\right)^\delta\right)'\right)'
=(1-\beta)\dfrac{1+xz}{1+yz}+\beta,
\end{align*}
or equivalently,

\begin{align}\label{eq-generalized:starlike-f/z:series}
\left(\frac{f(z)}{z}\right)^\delta
&=\dfrac{\delta^2}{\mu\nu z^{{\delta}/{\nu}}}\left(\int_0^z\dfrac{1}{\eta^{{\delta}/{\mu}-{\delta}/{\nu}+1}}
\left(\int_0^\eta\dfrac{1}{\omega^{1-{\delta}/{\mu}}}\left((1-\beta)\dfrac{1+x\omega}{1+y\omega}+\beta\right)
d\omega\right)d\eta\right)\\
&=\beta+(1\!-\!\beta)\left(\left(\dfrac{1\!+\!xz}{1\!+\!yz}\right)\ast\sum_{n=0}^\infty\dfrac{\delta^2z^n}{(\delta+n\nu)(\delta+n\mu)}\right).
\end{align}
If $A(z)$ is taken as $\displaystyle\left(\dfrac{f(z)}{z}\right)^{\delta}$, then using
\eqref{eq-Gener:convex-psi:munu}, \eqref{eq-Gener:convex-Phi:munu} and \eqref{eq-Gener:convex-Upsilon:munu},
in $A(z)$, $zA'(z)$ and $z(zA'(z))'$ respectively, gives
%
\begin{align}\label{eq-Gener:Convex-zf/z-Double:Deriv:Convol}
\left(z\left(z\left(\frac{f(z)}{z}\right)^\delta\right)^\prime\right)^\prime
=\left(\beta+(1-\beta)
\left(\dfrac{1+xz}{1+yz}\right)\right)\ast\Upsilon_{\mu,\nu}^\delta(z).
\end{align}
Since
\begin{align}\label{eq-Gener:Convex-Simplified-zf/z:Deriv}
\left(z\left(\dfrac{f(z)}{z}\right)^\delta\right)'=
(1-\delta)\left(\dfrac{f(z)}{z}\right)^\delta+\delta\left(\dfrac{f(z)}{z}\right)^\delta\left(\dfrac{zf'(z)}{f(z)}\right),
\end{align}
this gives
\begin{align}\label{eq-Gener:Convex-Simpli-zf/z:Double:Deriv}
\left(z\left(\dfrac{f(z)}{z}\right)^\delta\left(\dfrac{zf'(z)}{f(z)}\right)\right)^\prime&=
\left(1-\dfrac{1}{\delta}\right)\left(z\left(\dfrac{f(z)}{z}\right)^\delta\right)'+
\dfrac{1}{\delta}\left(z\left(z\left(\dfrac{f(z)}{z}\right)^\delta\right)^\prime\right)^\prime,
\end{align}
taking the logarithmic derivative on both sides of the integral operator
\eqref{eq-weighted-integralOperator} and differentiating further with a simple computation involving
%
%
%
%
\eqref{eq-Gener:Convex-zf/z-Double:Deriv:Convol}
gives
\newline
$\displaystyle
\left(\!z\!\left(\!\dfrac{z(F_\delta)'}{F_\delta}\!\right)\!\!\left(\dfrac{F_\delta}{z}\right)^\delta\right)^\prime
$
\begin{align*}
\!=\!(1\!-\!\beta)\!\left(\!\int_0^1\!\!\!\lambda(t)\!\left(\!\left(\!1\!-\!\dfrac{1}{\delta}\!\right)\!\Phi_{\mu,\nu}^\delta(tz)
\!+\!\left(\!\dfrac{1}{\delta}\!\right)\!\Upsilon_{\mu,\nu}^\delta(tz)\!\right)\!dt\!+\!\dfrac{\beta}{(1\!-\!\beta)}\right)
\ast\left(\dfrac{1+xz}{1+yz}\right).
\end{align*}
The Noshiro-Warschawski's Theorem (for details see\cite[Theorem
2.16]{DU}) states that the function
$z^{2\!-\!\delta}\!(F_\delta)^{\delta\!-\!1}\!(F_\delta)'$
defined in the region $\mathbb{D}$ is univalent if
$\left(z^{2-\delta}(F_\delta)^{\delta-1}(F_\delta)'\right)'$ is
contained in the half plane not containing the origin. Hence,
from the result based on duality principle \cite[Pg. 23]{Rus},
it follows that

$0\neq\left(z\left(\dfrac{z(F_\delta)'}{F_\delta}\right)\left(\dfrac{F_\delta}{z}\right)^\delta\right)^\prime$
\newline
is true if, and only if

\begin{align*}
{\rm Re}{\,}(1-\beta)\left(\int_0^1\lambda(t)\left(\left(1-\dfrac{1}{\delta}\right)\Phi_{\mu,\nu}^\delta(tz)
+\left(\dfrac{1}{\delta}\right)\Upsilon_{\mu,\nu}^\delta(tz)\right)dt+\dfrac{\beta}{(1-\beta)}\right)>\dfrac{1}{2}
\end{align*}
or equivalently,

\begin{align*}
{\rm Re}{\,}(1-\beta)\left(\int_0^1\lambda(t)\left(\left(1-\dfrac{1}{\delta}\right)\Phi_{\mu,\nu}^\delta(tz)
+\left(\dfrac{1}{\delta}\right)\Upsilon_{\mu,\nu}^\delta(tz)\right)dt+\dfrac{\beta-\frac{1}{2}}{(1-\beta)}\right)>0.
\end{align*}
Now, substituting \eqref{Beta-Cond-Generalized:Convex} in the
above inequality implies

\begin{align}\label{eq-Gener:Convex-Phi+Upsilon-q}
{\rm Re}{\,}\int_0^1\lambda(t)\left(\left(1-\dfrac{1}{\delta}\right)\Phi_{\mu,\nu}^\delta(tz)
+\left(\dfrac{1}{\delta}\right)\Upsilon_{\mu,\nu}^\delta(tz)- q_{\mu,\nu}^\delta(t)\right)dt>0.
\end{align}
From equation \eqref{eq-Gener:convex-Phi:munu} and
\eqref{eq-Gener:convex-Upsilon:munu}, it is easy to see that

\begin{align*}
\left(1-\dfrac{1}{\delta}\right)\Phi_{\mu,\nu}^\delta(tz)
+\left(\dfrac{1}{\delta}\right)\Upsilon_{\mu,\nu}^\delta(tz)
=\sum_{n=0}^\infty\dfrac{\delta (n+1)(n+\delta)(tz)^n}{(\delta+n\nu)(\delta+n\mu)},
\end{align*}
whose integral representation is given as

{\small{
\begin{align}\label{eq-generalized-Phi-Upsilon_munu-series}
\left(1\!-\!\dfrac{1}{\delta}\right)\Phi_{\mu,\nu}^\delta(tz)
\!+\!\left(\dfrac{1}{\delta}\right)\Upsilon_{\mu,\nu}^\delta(tz)=\dfrac{\delta^2}{\mu\nu}\!\int_0^1\!\!\!\int_0^1\!
\left(
\dfrac{\left(1\!-\!\frac{1}{\delta}\right)}{\left(1\!-\!trsz\right)^2}+\dfrac{\frac{1}{\delta}(1\!+\!trsz)}
{\left(1-trsz\right)^3}\right)r^{\delta/\nu-1}s^{\delta/\mu-1}drds.
\end{align}}}
Thus, using \eqref{eq-Gener:Convex-Integral:q:gamma>0} and
\eqref{eq-generalized-Phi-Upsilon_munu-series} in
\eqref{eq-Gener:Convex-Phi+Upsilon-q} and on further using the
fact that ${\rm Re\,} \left(\frac{1}{1-rstz}\right)^2\geq
\frac{1}{(1+rst)^2}$ for $z\in\mathbb{D}$, directly implies

{\small{
\begin{align*}
{\rm Re}{\,}\int_0^1&\lambda(t)\left(\!\int_0^1\!\!\!\int_0^1\!\left(\left(1\!-\!\dfrac{1}{\delta}\right)
\dfrac{1}{\left(1\!-\!trsz\right)^2}+\left(\dfrac{1}{\delta}\right)\dfrac{1\!+\!trsz}
{\left(1-trsz\right)^3}\right)r^{\delta/\nu-1}s^{\delta/\mu-1}drds\right.\\
&\left.-\!\int_0^1\!\!\!\int_0^1\!\left(\!\left(\!1\!-\!\dfrac{1}{\delta}\!\right)
\dfrac{1\!-\!\xi\left(1\!+\!trs\right)}{(1\!-\!\xi)\left(1\!
    +\!trs\right)^2}+\left(\dfrac{1}{\delta}\right)\dfrac{1\!-\!trs\!-\!\xi\left(1\!+\!trs\right)}
{(1\!-\!\xi)\left(1+trs\right)^3}\!\right)r^{\delta/\nu-1}s^{\delta/\mu-1}\!drds\!\right)\!dt
\end{align*}}}
{\small{
\begin{align*}
\geq\int_0^1&\lambda(t)\left(\!\int_0^1\!\!\!\int_0^1\!\left(\left(1\!-\!\dfrac{1}{\delta}\right)
\dfrac{1}{\left(1\!+\!trs\right)^2}+\left(\dfrac{1}{\delta}\right)\dfrac{1\!-\!trs}
{\left(1+trs\right)^3}\right)r^{\delta/\nu-1}s^{\delta/\mu-1}drds\right.\\
&\left.-\!\int_0^1\!\!\!\int_0^1\!\left(\!\left(\!1\!-\!\dfrac{1}{\delta}\!\right)
\dfrac{1\!-\!\xi\left(1\!+\!trs\right)}{(1\!-\!\xi)\left(1\!+\!trs\right)^2}
    +\left(\dfrac{1}{\delta}\right)\dfrac{1\!-\!trs\!-\!\xi\left(1\!+\!trs\right)}
{(1\!-\!\xi)\left(1+trs\right)^3}\!\right)r^{\delta/\nu-1}s^{\delta/\mu-1}\!drds\!\right)\!dt
\end{align*}}}
{\small{
\begin{align*}
=\int_0^1\!\!\!\lambda(t)\left(\!\int_0^1\!\!\!\int_0^1\!\left(\left(1\!-\!\dfrac{1}{\delta}\right)
\dfrac{\xi trs}{(1\!-\!\xi)\left(1\!+\!trs\right)^2}\!+\!\left(\dfrac{1}{\delta}\right)\dfrac{2\xi
trs}{(1\!-\!\xi)\left(1+trs\right)^3}\right)r^{\delta/\nu-1}s^{\delta/\mu-1}drds\right)\!dt
\end{align*}}}
{\small{
\begin{align*}
=\int_0^1\!\!\!\lambda(t)\left(\!\int_0^1\!\!\!\int_0^1\!\left(1+trs+\dfrac{1}{\delta}(1-trs)\right)\dfrac{\xi trs}
{(1\!-\!\xi)\left(1+trs\right)^3}r^{\delta/\nu-1}s^{\delta/\mu-1}drds\right)\!dt>0.
\end{align*}}}
Thus, ${\rm
Re}\left(z^{2-\delta}(F_\delta)^{\delta-1}(F_\delta)'\right)'>0$,
means that the function $z^{2-\delta}
(F_\delta)^{\delta-1}(F_\delta)'$ is univalent in $\mathbb{D}$.

In the next part of the theorem  the following two cases are
discussed to show the order of starlikeness condition for the
function $z^{2-\delta}(F_\delta)^{\delta-1}(F_\delta)'$.

\textbf{Case (i).} Let $\gamma=0 \,(\mu=0, \nu=\alpha\geq0)$.
The function $H(z)$ defined in
\eqref{eq-generalized:starlike-H} decreases to
\begin{align*}
H(z)
=\dfrac{\alpha}{\delta} z^{1-\delta/\alpha}\left(z^{\delta/\alpha}\left(\dfrac{f}{z}\right)^\delta\right)'.
\end{align*}
Thus using duality principle, it is easy to see that

\begin{align}\label{eq-Gener:convex-f/z-gamma0}
\left(\dfrac{f}{z}\right)^\delta=\beta+\dfrac{\delta(1-\beta)}{\alpha z^{\delta/\alpha}}
\int_0^z \omega^{\delta/\alpha-1}\left(\dfrac{1+x\omega}{1+y\omega}\right)d\omega,
\end{align}
where $|x|=|y|=1$ and $z\in\mathbb{D}$. A famous result from
the theory of convolution \cite[P. 94]{Rus} states that, if

\begin{align}\label{eq-gener:Convex-equiv1}
g\in\mathcal{S}^\ast(\xi)\Longleftrightarrow
\dfrac{1}{z}(g\ast h_\xi)(z)\neq0,
\end{align}
where $h_\xi(z)$ is defined in
\eqref{eq-h(z)-extremal-starlike}.

For the function
$f(z)\!\in\!\mathcal{W}_\beta^\delta(\alpha,0)$, the generalized
integral operator $F_\delta$ defined in
\eqref{eq-weighted-integralOperator}, belongs to the class
$\mathcal{C}_\delta(\zeta)$ with the conditions
$\left(1\!-\!\frac{1}{\delta}\right)\leq\zeta\leq\left(1-\frac{1}{2\delta}\right)$
and $\delta\geq1$, is equivalent of getting
$z\left(\frac{F_\delta}{z}\right)^\delta\left(\frac{z(F_\delta)'}{F_\delta}\right)\in\mathcal{S}^\ast(\xi)$,
where $\xi$ is defined by the hypothesis,
$\xi:=1-\delta(1-\zeta)$ and $0\leq\xi\leq1/2$. Therefore,
\eqref{iff:relation:convex+starlike-integral} and
\eqref{eq-gener:Convex-equiv1} leads to

\begin{align*}
z\left(\dfrac{F_\delta}{z}\right)^\delta\left(\dfrac{z(F_\delta)'}{F_\delta}\right)\in\mathcal{S}^\ast(\xi)\Longleftrightarrow
0\neq\dfrac{1}{z}\left(z\left(\dfrac{F_\delta}{z}\right)^\delta\left(\dfrac{z(F_\delta)'}{F_\delta}\right)\ast
h_\xi(z)\right).
\end{align*}
Further, using logarithmic derivative of \eqref{eq-weighted-integralOperator} in the
above expression gives
\begin{align}\label{eq-convex:gener:(1)}
0\neq
&
\displaystyle\int_0^1\lambda(t)\left(\dfrac{f(tz)}{tz}\right)^\delta\left(\dfrac{tzf'(tz)}{f(tz)}\right)
dt\ast \dfrac{h_\xi(z)}{z}
\nonumber\\
=
&
\displaystyle\int_0^1\dfrac{\lambda(t)}{1-tz}dt\ast\left(\dfrac{f(z)}{z}\right)^\delta\left(\dfrac{zf'(z)}{f(z)}\right)\ast
\dfrac{h_\xi(z)}{z}.
\end{align}
Now, using a simple computation involving $z(f/z)'$, it is easy to see that $\eqref{eq-convex:gener:(1)}$ is equivalent to 
\begin{align*}
0\neq&\displaystyle\int_0^1\dfrac{\lambda(t)}{1-tz}dt\ast\left(\left(1-\dfrac{1}{\delta}\right)
\left(\dfrac{f(z)}{z}\right)^\delta+\dfrac{1}{\delta}\left(z\left(\dfrac{f(z)}{z}\right)^\delta\right)'\right)\ast
\dfrac{h_\xi(z)}{z}\\
=&\displaystyle\int_0^1\lambda(t){\,}\left(\left(1-\dfrac{1}{\delta}\right)\dfrac{h_\xi(tz)}{tz}+
\dfrac{1}{\delta}h_\xi'(tz)\right){\,}dt\ast\left(\dfrac{f(z)}{z}\right)^\delta.
\end{align*}
Substituting the value of $(f/z)^\delta$ from
\eqref{eq-Gener:convex-f/z-gamma0} will give
{\small{
\begin{align*}
0\neq\left(\displaystyle\int_0^1\lambda(t){\,}\left(\left(1-\dfrac{1}{\delta}\right)\dfrac{h_\xi(tz)}{tz}+
\dfrac{1}{\delta}h_\xi'(tz)\right){\,}dt\right)\ast\left(\beta+\dfrac{\delta(1-\beta)}{\alpha z^{\delta/\alpha}}
\int_0^z
\omega^{\delta/\alpha-1}\left(\dfrac{1+x\omega}{1+y\omega}\right)d\omega\right)
\end{align*}}}

{\footnotesize{
\begin{align*}
\quad=(1-\beta)\left(\displaystyle\int_0^1\lambda(t)\left(\dfrac{\delta}{\alpha z^{\delta/\alpha}}
\int_0^z\omega^{\delta/\alpha-1}\left(\left(1-\dfrac{1}{\delta}\right)\dfrac{h_\xi(t\omega)}{t\omega}+
\dfrac{1}{\delta}h_\xi'(t\omega)\right)d\omega \right)dt+\dfrac{\beta}{1-\beta}\right)\ast\dfrac{1+xz}{1+yz}.
\end{align*}}}
Again from \cite[Pg. 23]{Rus} the above expression
is true if, and only if,

{\small{
\begin{align*}
{\rm Re}{\,}(1-\beta)\left(\displaystyle\int_0^1\lambda(t)\left(\dfrac{\delta}{\alpha z^{\delta/\alpha}}
\int_0^z\omega^{\delta/\alpha-1}\left(\left(1-\dfrac{1}{\delta}\right)\dfrac{h_\xi(t\omega)}{t\omega}+
\dfrac{1}{\delta}h_\xi'(t\omega)\right)d\omega \right)dt+\dfrac{\beta}{1-\beta}\right)>\dfrac{1}{2}
\end{align*}}}
or equivalently,

{\small{
\begin{align*}
{\rm Re}{\,}(1-\beta)\left(\displaystyle\int_0^1\lambda(t)\left(\dfrac{\delta}{\alpha z^{\delta/\alpha}}
\int_0^z\omega^{\delta/\alpha-1}\left(\left(1-\dfrac{1}{\delta}\right)\dfrac{h_\xi(t\omega)}{t\omega}+
\dfrac{1}{\delta}h_\xi'(t\omega)\right)d\omega \right)dt+\dfrac{\beta-\frac{1}{2}}{1-\beta}\right)>0.
\end{align*}}}
Using the condition on $\beta$ given in
\eqref{Beta-Cond-Generalized:Convex}, the above inequality
reduces to

\begin{align*}
{\rm Re}\displaystyle\int_0^1\lambda(t)\left(\dfrac{\delta}{\alpha z^{\delta/\alpha}}
\int_0^z\omega^{\delta/\alpha-1}\left(\left(1-\dfrac{1}{\delta}\right)\dfrac{h_\xi(t\omega)}{t\omega}+
\dfrac{1}{\delta}h_\xi'(t\omega)\right)d\omega -q_{0,\alpha}^\delta(t)\right)dt\geq0.
\end{align*}
Changing the variable $t\omega=u$, integrating by parts with respect to $t$ and on further
using \eqref{eq-generalized:convex-q:gamma0} and
\eqref{eq-weighted:Lambda}, the above inequality gives
%
%
\begin{align*}
{\rm Re}\displaystyle\int_0^1\Lambda_\alpha^\delta(t)\dfrac{d}{dt}\left(\dfrac{\delta}{\alpha z^{\delta/\alpha}}
\int_0^{tz}u^{\delta/\alpha-1}\left(\left(1-\dfrac{1}{\delta}\right)\dfrac{h_\xi(u)}{u}+
\dfrac{1}{\delta}h_\xi'(u)\right)du -t^{\delta/\alpha}q_{0,\alpha}^\delta(t)\right)dt\geq0
\end{align*}
or equivalently,

{\footnotesize{
\begin{align*}
{\rm Re}\int_0^1t^{{\delta}/{\alpha}-1}\Lambda_\alpha^\delta(t)\left[\left(1-\dfrac{1}{\delta}\right)
\left(\dfrac{h_\xi(tz)}{tz}-\dfrac{1-\xi(1+t)}{(1-\xi)(1+t)^2}\right)+\left(\dfrac{1}{\delta}\right)
\left(h_\xi'(tz)-\dfrac{1-t-\xi(1+t)}{(1-\xi)(1+t)^3}\right)\right]dt\geq0.
\end{align*}}}

\textbf{Case (ii).} Let $\gamma>0$ ($\mu>0$, $\nu>0$). Using
the conditions \eqref{iff:relation:convex+starlike-integral}
and \eqref{eq-gener:Convex-equiv1}, the integral transform
$V_\lambda^\delta(\mathcal{W}_\beta^\delta(\alpha,\gamma))\subset
\mathcal{C}_\delta(\zeta)$, for
$1-\frac{1}{\delta}\leq\zeta\leq\left(1-\frac{1}{2\delta}\right)$,
$\delta\geq1$ is equivalent of getting

\begin{align*}
0\neq\dfrac{1}{z}\left(z\left(\dfrac{F_\delta}{z}\right)^\delta\left(\dfrac{z(F_\delta)'}{F_\delta}\right)\ast
h_\xi(z)\right),
\end{align*}
where $\xi=1-\delta(1-\zeta)$ and $0\leq\xi\leq1/2$. Hence
using \eqref{eq-Gener:Convex-Simplified-zf/z:Deriv} and
\eqref{eq-convex:gener:(1)}, a simple computation similar to case (i) reduces the
above expression to
\begin{align*}
0\neq\displaystyle\int_0^1\lambda(t){\,}\left(\left(1-\dfrac{1}{\delta}\right)\dfrac{h_\xi(tz)}{tz}+
\dfrac{1}{\delta}h_\xi'(tz)\right){\,}dt\ast\left(\dfrac{f(z)}{z}\right)^\delta.
\end{align*}
Using \eqref{eq-Gener:Convex-Simplified-zf/z:Deriv} in the above inequality provides 
\begin{align*}
0\neq&\displaystyle\int_0^1\lambda(t){\,}\left(\left(1-\dfrac{1}{\delta}\right)\dfrac{h_\xi(tz)}{tz}+
\dfrac{1}{\delta}h_\xi'(tz)\right){\,}dt\ast\left[\beta+(1-\beta)\left(\dfrac{1+xz}{1+yz}\right)\right]\ast \psi_{\mu,\nu}^\delta(z)\\
=&(1-\beta)\left(\displaystyle\int_0^1\!\!\!\lambda(t)\left(\left(1-\dfrac{1}{\delta}\right)\dfrac{h_\xi(tz)}{tz}+
\dfrac{1}{\delta}h_\xi'(tz)\right){\,}dt+\dfrac{\beta}{1-\beta}\right)\ast
\psi_{\mu,\nu}^\delta(z)\ast\left(\dfrac{1+xz}{1+yz}\right).
\end{align*}
which is true if, and only if,
\begin{align*}
{\rm Re}{\,}(1-\beta)\left(\displaystyle\int_0^1\lambda(t){\,}\left(\left(1-\dfrac{1}{\delta}\right)\dfrac{h_\xi(tz)}{tz}+
\dfrac{1}{\delta}h_\xi'(tz)\right){\,}dt+\dfrac{\beta}{1-\beta}\right)\ast
\psi_{\mu,\nu}^\delta(z)>\dfrac{1}{2}
\end{align*}
or equivalently,
\begin{align*}
{\rm Re}{\,}(1-\beta)\left(\displaystyle\int_0^1\lambda(t){\,}
\left(\left(1-\dfrac{1}{\delta}\right)\dfrac{h_\xi(tz)}{tz}+
\dfrac{1}{\delta}h_\xi'(tz)\right){\,}dt+\dfrac{\beta-\frac{1}{2}}{1-\beta}\right)\ast
\psi_{\mu,\nu}^\delta(z)>0.
\end{align*}
Using the condition on $\beta$ given in
\eqref{Beta-Cond-Generalized:Convex}, the above inequality
becomes
\begin{align*}
{\rm Re}{\,}\displaystyle\int_0^1\lambda(t){\,}
\left(\left(1-\dfrac{1}{\delta}\right)\dfrac{h_\xi(tz)}{tz}+\dfrac{1}{\delta}h_\xi'(tz)-q_{\mu,\nu}^\delta(t)\right){\,}dt\ast
\psi_{\mu,\nu}^\delta(z)\geq0
\end{align*}
which on further using \eqref{eq-Gener:convex-psi:munu} leads
to
{\small{
\begin{align*}
{\rm Re}{\,}\displaystyle\int_0^1\lambda(t)
\left(\left(1-\dfrac{1}{\delta}\right)\dfrac{h_\xi(tz)}{tz}+
\dfrac{1}{\delta}h_\xi'(tz)-q_{\mu,\nu}^\delta(t)\right)dt\ast
\int_0^1\int_0^1\dfrac{1}{(1-u^{\nu/\delta/}v^{\mu/\delta}z)}{\,}dudv\geq0
\end{align*}}}
or equivalently,
{\small{
\begin{align*}
{\rm Re}\displaystyle\int_0^1\lambda(t)\left(\dfrac{\delta^2}{\mu\nu}
\int_0^1\int_0^1\left(\left(1-\dfrac{1}{\delta}\right)\dfrac{h_\xi(tzrs)}{tzrs}+
\dfrac{1}{\delta}h_\xi'(tzrs)\right)r^{\delta/\nu-1}s^{\delta/\mu-1}drds-q_{\mu,\nu}^\delta(t)\right)dt\geq0.
\end{align*}}}
Changing the variable $tr=\omega$,  
integrating with respect to $t$ and using \eqref{eq-weighted:Lambda} leads to
{\small{
\begin{align*}
{\rm Re}{\,}\displaystyle\int_0^1\!\!\!\!\Lambda_\nu^\delta(t)\dfrac{d}{dt}\!\left(\!\dfrac{\delta^2}{\mu\nu}\!
\int_0^t\!\!\!\int_0^1\!\!\left(\!\left(\!1\!-\!\dfrac{1}{\delta}\!\right)\dfrac{h_\xi(\omega zs)}{\omega zs}\!+\!
\dfrac{1}{\delta}h_\xi'(\omega zs)\!\right)\omega^{\delta/\nu-1}s^{\delta/\mu-1}ds d\omega-t^{\delta/\nu}q_{\mu,\nu}^\delta(t)\!\right)dt\!>\!0.
\end{align*}}}
Further, using \eqref{eq-Gener:Convex-q:gamma>0} reduces the
above inequality to
\begin{align*}
{\rm Re}{\,}\displaystyle\int_0^1\Lambda_\nu^\delta(t)t^{\delta/\nu-1}\!\left(
\int_0^1\left(\left(1-\dfrac{1}{\delta}\right)\left(\dfrac{h_\xi(stz)}{stz}-\dfrac{1-\xi(1+st)}
{(1-\xi)(1+st)^2}\right)\right.\right.\\
\left.\left.+\left(\dfrac{1}{\delta}\right)\left(h_\xi'(stz)-\dfrac{1-st-\xi(1+st)}{(1-\xi)(1+st)^3}\right)\right)
s^{\delta/\mu-1}ds\right)dt\geq0.
\end{align*}
Changing the variable $ts=\eta$, 
in the above expression, 
integrating with respect to $t$ and using
\eqref{eq-weighted:Pi} gives
\begin{align*}
{\rm Re}{\,}\displaystyle\int_0^1\Pi_{\mu,\nu}^\delta(t)\dfrac{d}{dt}\left(
\int_0^t\left(\left(1-\dfrac{1}{\delta}\right)\left(\dfrac{h_\xi(\eta z)}{\eta z}-\dfrac{1-\xi(1+\eta)}
{(1-\xi)(1+\eta)^2}\right)\right.\right.\\
\left.\left.+\left(\dfrac{1}{\delta}\right)\left(h_\xi'(\eta z)-\dfrac{1-\eta-\xi(1+\eta)}{(1-\xi)(1+\eta)^3}\right)\right)
\eta^{\delta/\mu-1}d\eta\right)dt\geq0
\end{align*}
or equivalently,

{\small{
\begin{align*}
{\rm Re}\!\displaystyle\int_0^1\!\!\Pi_{\mu,\nu}^\delta(t) t^{\delta/\mu-1}
\left[\!\left(1\!-\!\dfrac{1}{\delta}\right)\!\!\left(\!\dfrac{h_\xi(tz)}{tz}\!-\!\dfrac{1\!-\!\xi(1\!+\!t)}{(1\!-\!\xi)(1\!+\!t)^2}\!\right)
\!+\!\left(\dfrac{1}{\delta}\right)\!\!\left(\!h_\xi'(tz)\!-\!\dfrac{1\!-\!t\!-\!\xi(1\!+\!t)}{(1\!-\!\xi)(1\!+\!t)^3}\!\right)\!\right]\!dt\!\geq\!0
\end{align*}}}
which clearly implies that the function
$\mathcal{M}_{\Pi_{\mu,\nu}^\delta}(h_\xi)\geq0$ and the proof
is complete.

Now, to validate the condition of sharpness for the function
$f(z)\in\mathcal{W}_\beta^\delta(\alpha,\gamma)$, satisfying
the differential equation

\begin{align}\label{eq-generalized-f/z-extremal}
\dfrac{\mu\nu}{\delta^2}{\,} z^{1-\delta/\mu}\left(z^{\delta/\mu-\delta/\nu+1}
\left( z^{\delta/\nu}\left(\dfrac{f}{z}\right)^\delta\right)'\right)'
=\beta+(1-\beta)\dfrac{1+z}{1-z}
\end{align}
with the parameter $\beta<1$ defined in
\eqref{Beta-Cond-Generalized:Convex}. From
\eqref{eq-generalized-f/z-extremal}, a simple calculation gives

\begin{align}\label{eq-Gener:Convex-Shapness}
\left(\dfrac{f}{z}\right)^\delta
=1+2(1-\beta)\sum_{n=1}^\infty \dfrac{\delta^2z^n}{(\delta+n\nu)(\delta+n\mu)}.
\end{align}
Substituting \eqref{eq-Gener:Convex-Shapness} in
\eqref{eq-Gener:Convex-Simplified-zf/z:Deriv} will give

\begin{align}\label{eq-generalized-convex-simpli-f/z-1}
z\left(\dfrac{f}{z}\right)^\delta\left(\dfrac{zf'}{f}\right)
=z+2(1-\beta)\sum_{n=1}^{\infty}\dfrac{(n+\delta)\delta z^{n+1}}{(\delta+n\nu)(\delta+n\mu)}.
\end{align}
Further, substituting
\eqref{eq-generalized-convex-simpli-f/z-1} in the expression involving the logarithmic derivative of \eqref{eq-weighted-integralOperator} leads to
\begin{align}\label{eq-Gener:Convex-F:series}
z\left(\dfrac{F_\delta}{z}\right)^\delta\left(\dfrac{z(F_\delta)'}{F_\delta}\right)
&=\int_0^1\dfrac{\lambda(t)}{t}tz\left(\dfrac{f(tz)}{tz}\right)^\delta\left(\dfrac{tzf'(tz)}{f(tz)}\right)dt\nonumber\\
&=z+2(1-\beta)\sum_{n=1}^{\infty}\dfrac{(n+\delta)\delta\tau_n z^{n+1}}{(\delta+n\nu)(\delta+n\mu)}
\end{align}
where $\tau_n=\int_0^1t^n\lambda(t)dt$. Differentiating
\eqref{eq-Gener:Convex-F:series} will give

\begin{align*}
\left(z\left(\dfrac{F_\delta}{z}\right)^\delta\left(\dfrac{z(F_\delta)'}{F_\delta}\right)\right)'
=1+2(1-\beta)\sum_{n=1}^{\infty}\dfrac{(n+1)(n+\delta)\delta\tau_n z^{n}}{(\delta+n\nu)(\delta+n\mu)}.
\end{align*}
which clearly implies
\newline
$\displaystyle
z\left.\left(z\left(\dfrac{F_\delta}{z}\right)^\delta\left(\dfrac{z(F_\delta)'}{F_\delta}\right)\right)'\right|_{z=-1}
$
\begin{align}\label{eq-gener-Convex:Mainthm:(1)}
=-1-2(1-\beta)\sum_{n=1}^{\infty}\dfrac{(-1)^{n}(n+1-\xi)(n+\delta)\delta\tau_n }{(\delta+n\nu)(\delta+n\mu)}
\quad+2(1-\beta)\xi\sum_{n=1}^{\infty}\dfrac{(-1)^{n+1}(n+\delta)\delta\tau_n }{(\delta+n\nu)(\delta+n\mu)}.
\end{align}
The series expansion of the function $q_{\mu,\nu}^\delta(t)$
defined in \eqref{eq-Gener:Convex-Integral:q:gamma>0} is

\begin{align}\label{eq-Gener:Convex-q:series}
q_{\mu,\nu}^\delta(t)=1+\dfrac{\delta}{(1-\xi)}\sum_{n=1}^\infty\dfrac{(n+\delta)(n+1-\xi)(-1)^nt^n}{(\delta+n\nu)(\delta+n\mu)}.
\end{align}
whose representation in the form of generalized hypergeometric
function is given as

\begin{align}\label{eq-gener:convex-q:hyper:series}
q_{\mu,\nu}^\delta(t)=\,_5F_4\left(1,(1+\delta),(2-\xi),\dfrac{\delta}{\mu},\dfrac{\delta}{\nu};{\,}
\delta,(1-\xi),\left(1+\dfrac{\delta}{\mu}\right),\left(1+\dfrac{\delta}{\nu}\right);{\,}-t\right).
\end{align}
Using \eqref{eq-Gener:Convex-q:series} in
\eqref{Beta-Cond-Generalized:Convex} gives
\begin{align}\label{eq:Gener:Convex-beta-series}
\dfrac{\left(\beta-\frac{1}{2}\right)}{(1-\beta)}
=-1-\dfrac{\delta}{(1-\xi)}\sum_{n=1}^\infty\dfrac{(n+\delta)(n+1-\xi)(-1)^n\tau_n}{(\delta+n\nu)(\delta+n\mu)}.
\end{align}
From \eqref{eq-Gener:Convex-F:series} and
\eqref{eq:Gener:Convex-beta-series}, the expression
\eqref{eq-gener-Convex:Mainthm:(1)} is equivalent to

\begin{align*}
z\left.\left(z\left(\dfrac{F_\delta}{z}\right)^\delta\left(\dfrac{z(F_\delta)'}{F_\delta}\right)\right)'\right|_{z=-1}
=\xi {\,}\left.z\left(\dfrac{F_\delta}{z}\right)^\delta\left(\dfrac{z(F_\delta)'}{F_\delta}\right)\right|_{z=-1},
\end{align*}
which means that the result is sharp.
\end{proof}
\begin{remark}
\begin{enumerate}[1.]
\item For $\delta=1$ and $\xi=0$, {\rm Theorem
    \ref{Thm-Gener:Convex-Mainresult}} is similar to
    {\rm\cite[Theorem 3.1]{MahnazC}}.
\item For $\delta=1$, {\rm Theorem
    \ref{Thm-Gener:Convex-Mainresult}} reduces to
    {\rm\cite[Theorem 3.1]{SarikaC}}.
\end{enumerate}
\end{remark}
The condition $\mathcal{M}_{\Pi_{\mu,\nu}^\delta}(h_\xi)\geq 0$
derived in Theorem \ref{Thm-Gener:Convex-Mainresult} is
difficult to use, therefore a simpler sufficient condition is
presented in the next result.
%
%
%
\begin{theorem}\label{Thm:Main-Gener:Convex-Decreas}
Let $\mu\in[1/2,1]$, $\nu\geq1$ and
$\left(1-\frac{1}{\delta}\right)\leq\zeta\leq
\left(1-\frac{1}{2\delta}\right)$, where $\delta\geq1$. Let
$\beta<1$ satisfy \eqref{Beta-Cond-Generalized:Convex} and

\begin{align}\label{eq-Gener:Convex-Decre-Cond}
\dfrac{t^{{1}/{\mu}(\delta-1)}\left(\delta\left(1-\frac{1}{\mu}\right)\Pi_{\mu,\nu}^\delta(t)-
t\left(\Pi_{\mu,\nu}^{\delta}(t)\right)^\prime\right)}{(\log(1/t))^{3-2\delta(1-\zeta)}}
\end{align}
is decreasing on $(0,1)$. Then the function
$\mathcal{M}_{\Pi_{\mu,\nu}^\delta}(h_\xi)(z)\geq0$, where
$\xi=1-\delta(1-\zeta)$ and $0\leq\xi\leq1/2$.
\end{theorem}
\begin{proof}
Since the function

\begin{align*}
\mathcal{M}_{\Pi_{\mu,\nu}^\delta}(h_\xi)(z)=&\int_0^1t^{{\delta}/{\mu}-1}\Pi_{\mu,\nu}^\delta(t)
\left(\left(1-\dfrac{1}{\delta}\right)\left({\rm Re}\dfrac{h_\xi(tz)}{tz}+\dfrac{1-\xi(1+t)}{(1-\xi)(1+t)^2}\right)\right.\\
&\quad+\left.\left(\dfrac{1}{\delta}\right)
\left({\rm Re}{\,}h_\xi'(tz)-\dfrac{1-t-\xi(1+t)}{(1-\xi)(1+t)^3}\right)\right)dt,
\end{align*}
where $\xi=1-\delta(1-\zeta)$ and $0\leq\xi\leq1/2$.
Equivalently, it can also be written as

\begin{align*}
\mathcal{M}_{\Pi_{\mu,\nu}^\delta}(h_\xi)(z)=&\int_0^1 t^{{\delta}/{\mu}-1}\Pi_{\mu,\nu}^\delta(t)
\left(\left(1-\dfrac{1}{\delta}\right)\left({\rm Re}\dfrac{h_\xi(tz)}{tz}+\dfrac{1-\xi(1+t)}{(1-\xi)(1+t)^2}\right)\right.\\
&\left.\quad+\dfrac{d}{dt}\left(\dfrac{1}{\delta}\left({\rm Re}\dfrac{h_\xi(tz)}{z}-\dfrac{t(1-\xi(1+t))}{(1-\xi)(1+t)^2}\right)\right)\right)dt,
\end{align*}
which on further simplification gives

{\small{
\begin{align}\label{eq-Gener:Convex-Main-ineq}
\mathcal{M}_{\Pi_{\mu,\nu}^\delta}&(h_\xi)(z)\!=\!\left(\!1\!-\!\dfrac{1}{\delta}\!\right)\!\int_0^1\!
t^{{\delta}/{\mu}-1}\Pi_{\mu,\nu}^\delta(t)\left({\rm
Re}\dfrac{h_\xi(tz)}{tz}
\!+\!\dfrac{1\!-\!\xi(1\!+\!t)}{(1\!-\!\xi)(1\!+\!t)^2}\right)\!dt\nonumber\\
&\quad\quad+\!\int_0^1\!t^{{\delta}/{\mu}-1}\left(\dfrac{1}{\delta}\right)\left(\left(1\!-\!\dfrac{\delta}{\mu}\right)\Pi_{\mu,\nu}^\delta(t)\!
-\!t\left(\Pi_{\mu,\nu}^\delta(t)\right)'\right)\!\!\left(\!{\rm
Re}\dfrac{h_\xi(tz)}{tz}
\!+\!\dfrac{1\!-\!\xi(1\!+\!t)}{(1\!-\!\xi)(1\!+\!t)^2}\right)\!dt\nonumber\\
&=\int_0^1 t^{{\delta}/{\mu}-1}\left(\left(1\!-\!\frac{1}{\mu}\right)\Pi_{\mu,\nu}^\delta(t)\!-\!\left(\dfrac{1}{\delta}\right)
t\left(\Pi_{\mu,\nu}^{\delta}(t)\right)^\prime\right)
\left({\rm Re}\dfrac{h_\xi(tz)}{tz}-\dfrac{1-\xi(1+t)}{(1\!-\!\xi)(1+t)^2}\right)dt.
\end{align}}}
The right side of \eqref{eq-Gener:Convex-Main-ineq} is bounded
from below. So, due to the existence of lower bound, the
minimum principle states that, the minimum value of
\eqref{eq-Gener:Convex-Main-ineq} lies on the boundary i.e., on
$|z|=1$, where $z\neq1$. Now, minimizing ${\rm
Re}(h_\xi(tz)/(tz))$ with respect to $\epsilon$ will give

\begin{align*}
{\rm Re}\dfrac{h_\xi(tz)}{tz}\geq \dfrac{1}{2(1-\xi)}\left({\rm Re}\dfrac{2(1-\xi)+(2\xi-1)tz}{(1-tz)^2}-\dfrac{t}{|1-tz|^2}\right).
\end{align*}
Hence, \eqref{eq-Gener:Convex-Main-ineq} is equivalent of
obtaining

\begin{align*}
\int_0^1 t^{{\delta}/{\mu}-1}&\left(\delta\left(1-\dfrac{1}{\mu}\right)
\Pi_{\mu, \,\nu}^\delta(t)-t\left(\Pi_{\mu, \,\nu}^\delta(t)\right)^\prime\right)\\
&\left({\rm Re}\dfrac{2(1-\xi)+(2\xi-1)tz}{(1-tz)^2}
-\dfrac{t}{|1-tz|^2}-\dfrac{2(1-\xi(1+t))}{(1+t)^2}\right)dt\geq0.
\end{align*}
The equality of the above integral exist at $z=-1$. Since
$|z|=1$ and $z\neq1$, now letting ${\rm Re}z=y$ will reduce it
to considering
\begin{align*}
H_\Pi^{(\xi)}(y)\!=\!\int_0^1\! &t^{{\delta}/{\mu}-1}\left(\delta\left(1-\dfrac{1}{\mu}\right)
\Pi_{\mu, \,\nu}^\delta(t)-t\left(\Pi_{\mu, \,\nu}^\delta(t)\right)^\prime\right)\\
&\left(\!\dfrac{t(3\!-\!4(1\!\!+\!y)t\!+\!2(4y\!-\!1)t^2\!+\!4(y\!-\!1)t^3\!-\!t^4)}
{(1-2yt+t^2)^2(1+t)^2}\!-\!\dfrac{2\xi(1-t)}{(1\!-\!2yt\!+\!t^2)(1\!\!+\!t)}\!\right)\!dt\!\geq\!0
\end{align*}
where $|z|=1$ and $z\neq1$, gives $-1\leq y<1$. Since the term
$(1+y)\geq0$, $H_\Pi^{(\xi)}(y)$ can be written in
the series form as
\begin{align*}
H_\Pi^{(\xi)}(y)=\sum_{j=0}^\infty H_{j,\Pi}^{(\xi)}(1+y)^j,\quad\quad |1+y|<2.
\end{align*}
An easy computation shows that the $j$th term of
$H_{j,\Pi}^{(\xi)}$ is a positive multiple of

\begin{align*}
\tilde{H}_{j,\Pi}^{(\xi)}=\int_0^1t^{{\delta}/{\mu}-1}\left(\delta\left(1-\dfrac{1}{\mu}\right)
\Pi_{\mu, \,\nu}^\delta(t)-t\left(\Pi_{\mu, \,\nu}^\delta(t)\right)^\prime\right)(s_j(t)-2\xi u_j(t))dt,
\end{align*}
where
\begin{align*}
s_j(t):=\dfrac{(j+3)t^{j+1}}{(1+t)^{2j+4}}\left(1-2t+\dfrac{j-1}{j+3}t^2\right)\quad\quad{\rm and} \quad 
u_j(t):=\dfrac{t^{j+1}}{(1+t)^{2j+4}}(1-t^2).
\end{align*}
to give
\begin{align*}
s_j(t)-2\xi u_j(t)=\dfrac{t^{j+1}}{(1+t)^{2j+4}}\, v(t),
\end{align*}
with $v(t):=\left((j+3)(1-2t)+(j-1)t^2-2\xi(1-t^2)\right)$.

The function $v(t)$ is decreasing on $t\in(0,1)$. At $t=0$,
$v(t)$ is positive and at $t=1$, $v(t)$ is negative, which
clearly implies that the function $(s_j(t)-2\xi u_j(t))$ has
exactly one zero for $t\in(0,1)$. Set this zero by
$t_{j}^{(\xi)}$. Therefore, $(s_j(t)-2\xi u_j(t))>0$, for
$0\leq t<t_{j}^{(\xi)}$ and $(s_j(t)-2\xi u_j(t))<0$, for
$t_{j}^{(\xi)}<t<1$.

Now, define the functions

\begin{align}\label{eq-decre-H_j1}
\tilde{H}_{j}^{(\xi)}=\int_0^1 t^{{1}/{\mu}-1}(s_j(t)-2\xi u_j(t))\left(\log\left(\frac{1}{t}\right)\right)^{1+2\xi}dt
\end{align}
and

\begin{align*}
\tilde{\Pi}_{\mu,\nu}^{\delta,\xi}(t)=&t^{\frac{1}{\mu}(\delta-1)}\left(\delta\left(1-\frac{1}{\mu}\right)\Pi_{\mu,\nu}^\delta(t)-
t\left(\Pi_{\mu,\nu}^{\delta}(t)\right)^\prime\right)\\
&-\dfrac{(t_j^{(\xi)})^{\frac{1}{\mu}(\delta-1)}\left(\delta\left(\!1\!-\!\dfrac{1}{\mu}\!\right)\! \Pi_{\mu,
\,\nu}^\delta(t_j^{(\xi)})- t_j^{(\xi)}\left(\Pi_{\mu, \,\nu}^\delta(t_j^{(\xi)})\right)^\prime\right)}
{(\log(1/t_{j}^{(\xi)}))^{1+2\xi}}(\log(1/t))^{1+2\xi}.
\end{align*}
Since the hypothesis \eqref{eq-Gener:Convex-Decre-Cond} of the
theorem implies that the function

\begin{align*}
\dfrac{t^{\frac{1}{\mu}(\delta-1)}\left(\delta\left(1-\frac{1}{\mu}\right)\Pi_{\mu,\nu}^\delta(t)-
t\left(\Pi_{\mu,\nu}^{\delta}(t)\right)^\prime\right)}{(\log(1/t))^{1+2\xi}}
\end{align*}
is decreasing, where $\xi=1-\delta(1-\zeta)$ and
$0\leq\xi\leq1/2$, thus it is easy to observe that the
condition on $(s_j(t)-2\xi u_j(t))$ and the function
$\tilde{\Pi}_{\mu,\nu}^{\delta,\xi}(t)$ have same sign for
$t\in(0,1)$. Hence

\begin{align}\label{eq-decre-H_j2}
0&\leq\int_0^1 t^{\frac{1}{\mu}-1}\tilde{\Pi}_{\mu,\nu}^{\delta,\xi}(t)(s_j(t)-2\xi u_j(t))dt\nonumber \\
&=\tilde{H}_{j,\Pi}^{(\xi)}-\dfrac{(t_j^{(\xi)})^{\frac{1}{\mu}(\delta-1)}\left(\delta\left(1-\frac{1}{\mu}\right)\Pi_{\mu,
\,\nu}^\delta(t_j^{(\xi)})- t_j^{(\xi)}\left(\Pi_{\mu, \,\nu}^\delta(t_j^{(\xi)})\right)^\prime\right)}
{(\log(1/t_{j}^{(\xi)}))^{1+2\xi}}
\tilde{H}_{j}^{(\xi)}.
\end{align}
Using
\eqref{eq-weighted:Lambda} and \eqref{eq-weighted:Pi}, we have
\begin{align*}
\left(\Lambda_\nu^\delta(t)\right)'=-\lambda(t)t^{-\delta/\nu}\quad{\rm and}\quad
\left(\Pi_{\mu,\nu}^\delta(t)\right)'=-\Lambda_\nu^\delta(t)t^{-\delta/\mu+\delta/\nu-1}
\end{align*}
which clearly shows that
\begin{align*}
\dfrac{d}{dt}\left(\delta\left(1-\dfrac{1}{\mu}\right) \Pi_{\mu,\,\nu}^\delta(t)- t\left(\Pi_{\mu, \,\nu}^\delta(t)\right)^\prime\right)
&=\dfrac{d}{dt}\left(\delta\left(1-\dfrac{1}{\mu}\right) \Pi_{\mu,\,\nu}^\delta(t)+
t^{\delta/\nu-\delta/\mu}\Lambda_{\nu}^\delta(t)\right)\\
&=-\delta\left(1-\dfrac{1}{\nu}\right)t^{\delta/\nu-\delta/\mu-1}\Lambda_{\nu}^\delta(t)-t^{-\delta/\mu}\lambda(t)<0.
\end{align*}
for $\nu\geq1$ and $t\in(0,1)$. Thus, the above condition
implies
\begin{align*}
\delta\left(1-\dfrac{1}{\mu}\right) \Pi_{\mu,\,\nu}^\delta(t)- t\left(\Pi_{\mu, \,\nu}^\delta(t)\right)^\prime>0.
\end{align*}
Using similar arguments as in \cite[Page 280]{Aghalary} for the positivity of $\tilde{H}_{j}^{(\xi)}$ defined by
\eqref{eq-decre-H_j1}, from \eqref{eq-decre-H_j2}, it follows that
$\tilde{H}_{j,\Pi}^{(\xi)}\geq0$ and this completes the proof.
\end{proof}
\section{Applications of theorem \ref{Thm:Main-Gener:Convex-Decreas}}\label{Sec-Gener:Convex-Suff:Cond}

To apply Theorem \ref{Thm:Main-Gener:Convex-Decreas}, for the
case $\gamma>0$ $(\mu>0, \nu>0)$, it is required to show that
the function

\begin{align*}
\dfrac{t^{{(\delta-1)}/{\mu}}\left(\delta\left(1-{1}/{\mu}\right)\Pi_{\mu,\nu}^\delta(t)-
t\left(\Pi_{\mu,\nu}^{\delta}(t)\right)^\prime\right)}{(\log(1/t))^{3-2\delta(1-\zeta)}}
\end{align*}
is decreasing in the range $t\in(0,1)$, where $\mu\in[1/2,1]$,
$\nu\geq1$, $\delta\geq1$ and
$\left(1-\frac{1}{\delta}\right)\leq\zeta\leq
\left(1-\frac{1}{2\delta}\right)$. Since
$\xi=(1-\delta(1-\zeta))$, thus using \eqref{eq-weighted:Pi},
the above expression can be rewritten as

\begin{align*}
g(t):=\dfrac{\delta\left(1-\frac{1}{\mu}\right)t^{{\delta}/{\mu}-1/\mu}\,\Pi_{\mu,\nu}^\delta(t)+
t^{\delta/\nu-1/\mu}\,\Lambda_{\nu}^{\delta}(t)}{(\log(1/t))^{1+2\xi}},
\end{align*}
where $\xi\in[0,1/2]$. Note that the chosen function
$\lambda(t)$ satisfy the condition $\lambda(1)=0$. Therefore,
in the overall discussion, the assumed conditions hold.

Taking the derivative of $g(t)$ and using
\eqref{eq-weighted:Lambda} and \eqref{eq-weighted:Pi} will give

\begin{align*}
g'(t)=\dfrac{t^{{\delta}/{\mu}-1/\mu-1}h(t)}{\left(\log\frac{1}{t}\right)^{2(1+\xi)}}&
\left[\delta\left(1-\frac{1}{\mu}\right)\Pi_{\mu,\nu}^\delta(t)+
\left(1+\delta\left(\frac{1}{\nu}-1\right)\dfrac{\log\frac{1}{t}}{h(t)}\right)\,t^{\delta/\nu-\delta/\mu}\Lambda_{\nu}^{\delta}(t)\right.\\
&\left.-t^{1-\delta/\mu}\frac{\log\frac{1}{t}}{h(t)}\lambda(t)\right],
\end{align*}
where the function
$h(t):=\frac{1}{\mu}(\delta-1)\log\frac{1}{t}+(1+2\xi)$, which
by simple computation for $0<t<1$, $\delta\geq1$ and
$0\leq\xi\leq1/2$ gives $h(t)\geq1$.

Therefore, proving $g'(t)\leq0$ is equivalent of getting
$k(t)\leq0$, where

\begin{align*}
k(t):=\delta\left(1-\frac{1}{\mu}\right)\Pi_{\mu,\nu}^\delta(t)+
\left(1+\delta\left(\frac{1}{\nu}-1\right)\dfrac{\log\frac{1}{t}}{h(t)}\right)t^{\delta/\nu-\delta/\mu}\,\Lambda_{\nu}^{\delta}(t)
-t^{1-\delta/\mu}\frac{\log\frac{1}{t}}{h(t)}\lambda(t).
\end{align*}
Clearly $k(1)=0$ implies that if $k(t)$ is increasing function
of $t\in(0,1)$ then $g'(t)\leq0$. Hence, it is required to show
that

\begin{align*}
k'(t)=t^{\delta/\nu-{\delta}/{\mu}-1}\dfrac{l(t)}{h(t)},
\end{align*}
where

\begin{align*}
l(t):=&\left(\frac{\delta}{\nu}-\delta\right)\Lambda_{\nu}^{\delta}(t)
\left[\left(\frac{\delta}{\nu}-\frac{1}{\mu}\right)\log\frac{1}{t}+1+2\xi-
\dfrac{(1+2\xi)}{h(t)}\right]\\
&+t^{1-\delta/\nu}\lambda(t)\left[\left(\frac{1}{\mu}-\frac{\delta}{\nu}+\delta-1\right)\log\frac{1}{t}-1-2\xi
+\dfrac{(1+2\xi)}{h(t)}\right]-t^{2-\delta/\nu}\log\frac{1}{t}\lambda'(t)\geq0.
\end{align*}
Now, using the hypothesis $\lambda(1)=0$ implies that $l(1)=0$.
Therefore $l(t)$ is decreasing function of $t\in(0,1)$, i.e.,
if $l'(t)\leq0$, clearly means that the function $g(t)$ is
decreasing. Now, we calculate

{\small{
\begin{align*}
l'(t)=&\delta\left(1-\frac{1}{\nu}\right)\frac{\Lambda_{\nu}^{\delta}(t)}{t}
\left[\left(\frac{\delta}{\nu}-\frac{1}{\mu}\right)+\left(\frac{\delta}{\mu}-\frac{1}{\mu}\right)\dfrac{(1+2\xi)}{(h(t))^2}\right]
+t^{-\delta/\nu}\lambda(t)\left[(\delta-1)\left(1-\frac{1}{\mu}\right)\log\frac{1}{t}\right.\\
&\left.+\left(\frac{\delta}{\nu}-\frac{1}{\mu}+2\xi(\delta-1)\right)-\frac{(\delta-1)(1+2\xi)}{h(t)}
+\left(\frac{\delta}{\mu}-\frac{1}{\mu}\right)\frac{(1+2\xi)}{(h(t))^2}\right]\nonumber\\
&+t^{1-\delta/\nu}\lambda'(t)\left[\left(\frac{1}{\mu}+\delta-3\right)\log\frac{1}{t}-2\xi+\frac{(1+2\xi)}
{h(t)}\right]-\log\frac{1}{t}t^{2-\delta/\nu}\lambda''(t).
\end{align*}}}
Thus, the function $g'(t)\leq0$ is counterpart of the following
inequalities:

\begin{align}\label{eq-Gener:convex-Main:equiv(Cond1)}
\Lambda_{\nu}^{\delta}(t)\left[\left(\frac{1}{\mu}-\frac{\delta}{\nu}\right)(h(t))^2-(1+2\xi)\left(\frac{\delta}{\mu}-\frac{1}{\mu}\right)\right]\geq0
\end{align}
and

{\small{
\begin{align}\label{eq-Gener:convex-Main:equiv(Cond2)}
\lambda(t)\left[\left(\frac{\delta}{\nu}-\frac{1}{\mu}+2\xi(\delta-1)\right)+
(\delta-1)\left(1-\frac{1}{\mu}\right)\log\frac{1}{t}-\frac{(\delta-1)(1+2\xi)}{h(t)}
+\left(\frac{\delta}{\mu}-\frac{1}{\mu}\right)\frac{(1+2\xi)}{(h(t))^2}\right]\nonumber\\
+t\lambda'(t)\left[\left(\frac{1}{\mu}+\delta-3\right)\log\frac{1}{t}-2\xi+\frac{(1+2\xi)}
{h(t)}\right]-\log\frac{1}{t}t^{2}\lambda''(t)\leq0,
\end{align}}}
\hspace{-.10cm}for $\nu\geq1$ and $t\in(0,1)$. Letting
${(2-\delta)}/{\mu}\geq{\delta}/{\nu}$ implies that the
inequality \eqref{eq-Gener:convex-Main:equiv(Cond1)} is true,
which clearly means that the function $g(t)$ is decreasing, if
the inequality \eqref{eq-Gener:convex-Main:equiv(Cond2)} holds
along with the condition
${(2-\delta)}/{\mu}\geq{\delta}/{\nu}$, for $1\leq\delta<2$,
$\mu\in[1/2,1]$ and $\nu\geq1$.

The function $h(t)\geq1$ and
$\left(1-\delta/\mu\right)+\left(\delta/\mu-1/\mu\right)/h(t)\leq0$,
for $1/2\leq\mu\leq1$ and $\delta\geq1$. Thus the inequality
\eqref{eq-Gener:convex-Main:equiv(Cond2)} is true when

\begin{align}\label{eq-Gener:convex-Main:equiv(Cond3)}
\lambda(t)\left[\left(\frac{1}{\mu}-\frac{\delta}{\nu}-2\xi(\delta-1)\right)+
(\delta-1)\left(\frac{1}{\mu}-1\right)\log\frac{1}{t}+\left(\delta-\dfrac{\delta}{\mu}\right)\frac{(1+2\xi)}{h(t)}\right]\nonumber\\
+t\lambda'(t)\left[\left(3-\delta-\frac{1}{\mu}\right)\log\frac{1}{t}+2\xi-\frac{(1+2\xi)}{h(t)}\right]+\log\frac{1}{t}t^{2}\lambda''(t)\geq0.
\end{align}

In order to use the above condition for the application purposes, we consider the following.
For the parameters $A,B,C>0$, set
\begin{align}\label{eq-Generalized-hypergeometric-fn}
\lambda(t)=K t^{B-1}(1-t)^{C-A-B}\omega(1-t),
\end{align}
where the function
\begin{align*}
\omega(1-t)=1+\displaystyle\sum_{n=1}^\infty x_n(1-t)^n,\quad{\rm with}
\quad x_n\geq0,\quad t\in(0,1).
\end{align*}
The constant $K$ is chosen such that it satisfies normalization
condition $\int_0^1\lambda(t)dt=1$ and $(C-A-B)>0$ which
clearly implies that the function $\lambda(t)$ is zero at
$t=1$.

By an easy calculation, we get
{\small{
\begin{align}\label{eq-Generalized-hypergeometric-fn-lambda'}
\lambda'(t)=Kt^{B-2}(1-t)^{C-A-B-1}\left[\left(\frac{}{}(B-1)(1-t)-(C-A-B)t\right)\omega(1-t)-t(1-t)\omega'(1-t)\right],
\end{align}}}
and
{\small{
\begin{align}\label{eq-Generalized-hypergeometric-fn-lambda''}
\lambda^{\prime\prime}(t)=&Kt^{B-3}(1-t)^{C-A-B-2}\left[\left(\dfrac{}{}(B-1)(B-2)(1-t)^2\right.\right.\nonumber\\
&\left.\dfrac{}{}-2(B-1)(C-A-B)t(1-t)+(C-A-B)(C-A-B-1)t^2\right)\omega(1-t)\nonumber\\
&\left.+\left(\dfrac{}{}2(C-A-B)t-2(B-1)(1-t)\right)t(1-t)\omega'(1-t)+t^2(1-t)^2\omega''(1-t)\right].
\end{align}}}
Now, substituting the values of $\lambda(t)$, $\lambda'(t)$ and
$\lambda''(t)$ given in
\eqref{eq-Generalized-hypergeometric-fn},
\eqref{eq-Generalized-hypergeometric-fn-lambda'} and
\eqref{eq-Generalized-hypergeometric-fn-lambda''}, respectively
in inequality \eqref{eq-Gener:convex-Main:equiv(Cond3)} will
give the corresponding condition as

\begin{align}\label{eq-Gener:convex-Main-Gener:hyper-gamma>0}
t^2(1-t)^2\log\dfrac{1}{t}\omega''(1-t)+t(1-t)X_1(t)\omega'(1-t)+X_2(t)\omega(1-t)\geq0
\end{align}
where

{\small{
\begin{align*}
X_1(t):=\log\dfrac{1}{t}\left[(1-t)\left(\dfrac{1}{\mu}+\delta-2B-1\right)+2(C-A-B)t\right]
+(1-t)\left(-2\xi+\dfrac{(1+2\xi)}{h(t)}\right).
\end{align*}}}
and

{\footnotesize{
\begin{align*}
X_2(t):=&\log\dfrac{1}{t}\left[(1-t)^2\left[(\delta-1)\left(\frac{1}{\mu}-1\right)+(1-B)\left(\frac{1}{\mu}+\delta-B-1\right)\right]
+(C-A-B)t\times\right.\\
&\left.\left[(1-t)\left(\dfrac{1}{\mu}+\delta-2B-1\right)+(C-A-B-1)t\right]\right]+(1-t)\left[(1-t)
\left[\left(\frac{1}{\mu}\!-\!\frac{\delta}{\nu}\!-\!2\xi(\delta\!-\!B)\right)\right.\right.\\
&\left.\left.+\left(\delta+1-B-\frac{\delta}{\mu}\right)\dfrac{(1+2\xi)}{h(t)}\right]+(C-A-B)t
\left[-2\xi+\dfrac{(1+2\xi)}{h(t)}\right]\right].
\end{align*}}}
Since the function $\omega(1-t)=1+\sum_{n=1}^\infty
x_n(1-t)^n$, with the condition $x_n\geq0$, which clearly means
that the function $\omega(1-t)$, $\omega'(1-t)$ and
$\omega''(1-t)$ are non-negative for all values of $t\in(0,1)$.
Therefore, proving inequality
\eqref{eq-Gener:convex-Main-Gener:hyper-gamma>0}, it suffice to
show

\begin{align*}
X_1(t)\geq0\quad{\rm and }\quad X_2(t)\geq0.
\end{align*}
Now, in this respect the following two cases are examined:

\noindent {\bf{Case (i)}} Let $0<B\leq\delta$. By a simple adjustment, it can be easily obtained that the
    inequality $X_1(t)\geq0$ holds true if
\begin{align*}
\log\dfrac{1}{t}\left[(1-t)\left(\dfrac{1}{\mu}+\delta-2B-1\right)+2(C-A-B)t\right]
\geq2\xi(1-t),
\end{align*}
where $\xi=1-\delta(1-\zeta)$, for
$\left(1-\frac{1}{\delta}\right)\leq\zeta\leq
\left(1-\frac{1}{2\delta}\right)$. Since the right side of
the above inequality is positive for $\xi\in[0,1/2]$ and
$t\in(0,1)$, hence on using the condition
\begin{align}\label{eq-general:cond-log-1-t}
(1-t)\leq\dfrac{(1+t)}{2}\log\dfrac{1}{t},\quad t\in(0,1),
\end{align}
it is enough to get
\begin{align}\label{eq-Gener:convex-Main:ineq-gamma>0(2)}
\left(\dfrac{1}{\mu}+\delta-1-2B-\xi\right)(1-t)+2(C-A-B-\xi)t\geq0.
\end{align}
Further, the equivalent condition for $X_2(t)\geq0$ is
obtained. By the assumed hypothesis
${(2-\delta)}/{\mu}\geq{\delta}/{\nu}$ directly implies
$1/\mu\geq\delta/\nu$. Now using this condition, the
function $X_2(t)\geq0$ is valid if
{\small{
\begin{align*}
&\log\dfrac{1}{t}\left((C-A-B)t\left[(1-t)\left(\dfrac{1}{\mu}+\delta-2B-1\right)+(C-A-B-1)t\right]\right.\\
&\left.+(1-t)^2\left[(\delta-1)\left(\frac{1}{\mu}-1\right)+(1-B)\left(\frac{1}{\mu}+\delta-B-1\right)\right]\right)
+(1-t)\left(\dfrac{}{}(1-t)\times\right.\\
&\left.\left[-2\xi(\delta-B)-\left(\frac{1}{\mu}+B-1\right)\dfrac{(1+2\xi)}{h(t)}\right]-2\xi(C-A-B)t\right)\geq0.
\end{align*}}}
or equivalently,

{\small{
\begin{align}\label{eq-Gener:convex-Main-Gener:hyper-gamma>0-1}
&\log\dfrac{1}{t}\left((C-A-B)t\left[(1-t)\left(\dfrac{1}{\mu}+\delta-2B-1\right)+(C-A-B-1)t\right]\right.\nonumber\\
&\quad\left.+(1-t)^2\left[(\delta-1)\left(\frac{1}{\mu}-1\right)+(1-B)\left(\frac{1}{\mu}+\delta-B-1\right)\right]\right)\nonumber\\
\geq&(1-t)\left(2\xi(C-A-B)t+(1-t)\left[2\xi(\delta-B)+\left(\frac{1}{\mu}+B-1\right)\dfrac{(1+2\xi)}{h(t)}\right]\right).
\end{align}}}
As $0\leq B\leq\delta$, therefore using the conditions
$0\leq\xi\leq1/2$, $1/2\leq\mu\leq1$, and $(C-A-B)>0$, it
is easy to check that the coefficient of $(1-t)$ on right
side of the above expression is positive. Therefore, in
view of the inequality \eqref{eq-general:cond-log-1-t}, the
condition
\eqref{eq-Gener:convex-Main-Gener:hyper-gamma>0-1} holds
true for $t\in(0,1)$ if
\begin{align*}
&2\left((C-A-B)t\left[(1-t)\left(\dfrac{1}{\mu}+\delta-2B-1\right)+(C-A-B-1)t\right]\right.\nonumber\\
&\quad\left.+(1-t)^2\left[(\delta-1)\left(\frac{1}{\mu}-1\right)+(1-B)\left(\frac{1}{\mu}+\delta-B-1\right)\right]\right)\nonumber\\
\geq&(1+t)\left(2\xi(C-A-B)t+(1-t)\left[2\xi(\delta-B)+\left(\frac{1}{\mu}+B-1\right)\dfrac{(1+2\xi)}{h(t)}\right]\right)
\end{align*}
or equivalently,
\begin{align}\label{eq-Gener:convex-Main:ineq-gamma>0(1)}
(1-t)^2\left[2(\delta-1)\left(\frac{1}{\mu}-1\right)+2(1-B)\left(\frac{1}{\mu}+\delta-B-1\right)+R(t)\right]\nonumber\\
+2t(1-t)\left[(C-A-B)\left(\dfrac{1}{\mu}+\delta-2B-1-\xi\right)+R(t)\right]\nonumber\\
+2t^2(C-A-B)(C-A-B-1-2\xi)\geq0,
\end{align}
where
\begin{align*}
R(t):=\left(-\frac{1}{\mu}+1-B\right)\dfrac{(1+2\xi)}{h(t)}-2\xi(\delta-B).
\end{align*}
Consequently, the condition
\eqref{eq-Gener:convex-Main:ineq-gamma>0(1)} holds good if
the coefficients of $t^2$, $t(1-t)$, and $(1-t)^2$ are
positive. Now it remains to prove the following
inequalities:
\begin{align}\label{eq-Gener:convex-Main:ineq-gamma>0(11)}
(C-A-B)(C-A-B-1-2\xi)\geq0,
\end{align}
\begin{align}\label{eq-Gener:convex-Main:ineq-gamma>0(12)}
(C-A-B)\left(\dfrac{1}{\mu}+\delta-2B-1-\xi\right)+R(t)\geq0,
\end{align}
and
\begin{align}\label{eq-Gener:convex-Main:ineq-gamma>0(13)}
2(\delta-1)\left(\frac{1}{\mu}-1\right)+2(1-B)\left(\frac{1}{\mu}+\delta-B-1\right)+R(t)\geq0,
\end{align}
where $\xi=1-\delta+\delta\zeta$,
$\left(1-\frac{1}{\delta}\right)\leq\zeta\leq
\left(1-\frac{1}{2\delta}\right)$ and $\delta\geq1$.

\noindent {\bf{Case (ii)}} Consider the case when $B\geq\delta$. It is easy to observe that the
condition \eqref{eq-Gener:convex-Main:ineq-gamma>0(2)} is true when
\begin{align*}
\left(\dfrac{1}{\mu}+\delta-1-\xi\right)\geq2B,
\end{align*}
which clearly implies when $B\leq\delta$. Hence this case is not valid.

With the availability of the conditions given above we prove the result for 
the case $\gamma>0$ $(\mu>0,\nu>0)$ and $\lambda(t)$ defined in
\eqref{eq-Generalized-hypergeometric-fn}.
\begin{theorem}\label{Thm-Gener:Convex-Generalized-hypergeometric-fn:gamma>0}
Let $A,B,C>0$, $1/2\leq\mu\leq1\leq\nu$ and
$1-\frac{1}{\delta}\leq\zeta\leq1-\frac{1}{2\delta}$, for
$1\leq\delta\leq2$. Let $\beta<1$ satisfy
\begin{align*}
\dfrac{\beta-\frac{1}{2}}{1-\beta}=-K\int_0^1 t^{B-1}(1-t)^{C-A-B}\omega(1-t)
q_{\mu,\nu}^\delta(t) dt,
\end{align*}
where $q_{\mu,\nu}^\delta(t)$ is defined by the differential
equation \eqref{eq-Gener:Convex-q:gamma>0}, the constant $K$
and the function $\omega(1-t)$ is given in
\eqref{eq-Generalized-hypergeometric-fn}. Then for
$f(z)\in\mathcal{W}_\beta^\delta(\alpha,\gamma)$, the function
\begin{align*}
H_{A,\,B,\,C}^\delta(f)(z)=\left(K\int_0^1 t^{B-1}(1-t)^{C-A-B}\omega(1-t) \left(\frac{f(tz)}{t}\right)^\delta dt\right)^{1/\delta}
\end{align*}
belongs to $\mathcal{C}_\delta(\zeta)$ for the condition
${(2-\delta)}/{\mu}\geq{\delta}/{\nu}$, if
{\small{
\begin{align*}
C\geq A+B+2\quad{\rm and}\quad
B\leq\min\left\{\dfrac{1}{4}\left(\dfrac{1}{\mu}-3+\delta(3-2\zeta)\right)\,,\,
\dfrac{2}{\left(\delta+1/\mu\right)}\left(\dfrac{(2\delta-1)}{\mu}-\delta+1\right)\right\}.
\end{align*}}}
\end{theorem}
\begin{proof}
In order to prove the result, it is enough to get the
inequalities \eqref{eq-Gener:convex-Main:ineq-gamma>0(2)},
\eqref{eq-Gener:convex-Main:ineq-gamma>0(11)},
\eqref{eq-Gener:convex-Main:ineq-gamma>0(12)} and
\eqref{eq-Gener:convex-Main:ineq-gamma>0(13)} by using the
above hypothesis.

The inequalities \eqref{eq-Gener:convex-Main:ineq-gamma>0(2)}
and \eqref{eq-Gener:convex-Main:ineq-gamma>0(11)} are true if
$(C-A-B)\geq1+2\xi$ and $2B\leq(1/\mu+\delta-1-\xi)$, where
$\xi=1-\delta(1-\zeta)$. Since the parameters $(C-A-B)>2$ and
$4B\leq(1/\mu+\delta-1-2\xi)$, directly implies that these two
inequalities hold. Moreover, to show the existence of
inequality \eqref{eq-Gener:convex-Main:ineq-gamma>0(12)} under
the given hypothesis, it is enough to prove
\begin{align*}
(C-A-B)\left(\dfrac{1}{\mu}+\delta-2B-1-\xi\right)\geq \left(\frac{1}{\mu}+\delta-1\right)
\end{align*}
or equivalently,
\begin{align*}
(C-A-B-2)\left(\dfrac{1}{\mu}+\delta-2B-1-\xi\right)+\left(\frac{1}{\mu}+\delta-1-2\xi-4B\right)\geq0,
\end{align*}
that can be shown easily. Finally, to prove inequality
\eqref{eq-Gener:convex-Main:ineq-gamma>0(13)}, it is sufficient
to get
\begin{align*}
2(\delta-1)\left(\frac{1}{\mu}-1\right)+2(1-B)\left(\frac{1}{\mu}+\delta-B-1\right)\geq\left(\frac{1}{\mu}+\delta-1\right).
\end{align*}
By simple computation, the above condition is true if
\begin{align*}
\left(\frac{2\delta}{\mu}-\dfrac{1}{\mu}-\delta+1\right)-2B\left(\delta+\frac{1}{\mu}\right)\geq0,
\end{align*}
which is clearly true. Hence by the given hypothesis and
Theorem \ref{Thm:Main-Gener:Convex-Decreas}, the result
directly follows.
\end{proof}
So far, the case $\gamma>0$ was discussed in detail. Now, to
apply Theorem \ref{Thm-Gener:Convex-Mainresult} for the case
$\gamma=0$ $(\mu=0,\,\nu=\alpha\geq0)$, it is required to show
that the function

\begin{align*}
a(t):=\dfrac{\delta\left(1-\frac{1}{\alpha}\right)t^{{(\delta-1)}/{\alpha}}\Lambda_{\alpha}^\delta(t)+
t^{1-1/\alpha}\lambda(t)}{(\log(1/t))^{1+2\xi}}
\end{align*}
is decreasing on $t\in(0,1)$, where $\xi=1-\delta(1-\zeta)$,
for $1/2\leq\alpha\leq1$, $0\leq\xi\leq1/2$ and $\delta\geq1$.
Now, differentiating $a(t)$ and on using
\eqref{eq-weighted:Lambda} will give

\begin{align*}
a'(t)=\dfrac{p(t)\,t^{\delta/\alpha-1/\alpha-1}}{(\log(1/t))^{2+2\xi}}\,\,b(t),
\end{align*}
where

\begin{align*}
&b(t):=\delta\left(1-\dfrac{1}{\alpha}\right)\Lambda_\alpha^\delta(t)+\left(1-\dfrac{(\delta-1)\log(1/t)}{p(t)}\right)
t^{1-\delta/\alpha}\lambda(t)+\dfrac{\log(1/t)}{p(t)}t^{2-\delta/\alpha}\lambda'(t)\\
{\rm and}&\\
&p(t):=\frac{1}{\alpha}(\delta-1)\log\frac{1}{t}+(1+2\xi).
\end{align*}
When $\delta\geq1$, $\alpha\in[1/2,1]$ and $\xi\in[0,1/2]$, it
can be easily seen that the function $p(t)\geq1$, for
$t\in(0,1)$. Hence, proving $a'(t)\leq0$ is equivalent of
getting $b(t)\leq0$. Assuming $\lambda(1)=0$ will give
$b(1)=0$. Hence, if $b(t)$ is increasing function of
$t\in(0,1)$, then $a'(t)\leq0$ and this completes the proof.
Now

\begin{align*}
b'(t)=\dfrac{t^{-\delta/\alpha}}{p(t)}&\left[(\delta-1)\lambda(t)\left(\left(\dfrac{1}{\alpha}-1\right)
\log\frac{1}{t}-1-2\xi+\dfrac{(1+2\xi)}{p(t)}\right)\right.\\
&\quad\left.+t\lambda'(t)\left(\left(3-\delta-\dfrac{1}{\alpha}\right)\log\frac{1}{t}+1+2\xi-\dfrac{(1+2\xi)}{p(t)}\right)
+\log\dfrac{1}{t}\,\,t^2\lambda''(t)\right].
\end{align*}
Therefore, $b'(t)\geq0$, if

\begin{align}\label{eq-Gener:convex-Main:equiv(Cond3)gamma0}
&(\delta-1)\lambda(t)\left(\left(\dfrac{1}{\alpha}-1\right)
\log\frac{1}{t}-1-2\xi+\dfrac{(1+2\xi)}{p(t)}\right)\nonumber\\
&+t\lambda'(t)\left(\left(3-\delta-\dfrac{1}{\alpha}\right)\log\frac{1}{t}+1+2\xi-\dfrac{(1+2\xi)}{p(t)}\right)
+\log\dfrac{1}{t}\,\,t^2\lambda''(t)\geq0.
\end{align}
Now, using $\lambda(t)$, $\lambda'(t)$ and $\lambda''(t)$ given
in \eqref{eq-Generalized-hypergeometric-fn},
\eqref{eq-Generalized-hypergeometric-fn-lambda'} and
\eqref{eq-Generalized-hypergeometric-fn-lambda''}, respectively
in inequality \eqref{eq-Gener:convex-Main:equiv(Cond3)gamma0},
will give the corresponding condition as
\begin{align}\label{eq-Gener:convex-MainCond-Gener:hyper-gamma0}
t^2(1-t)^2\,\log\dfrac{1}{t}\,\,\omega''(1-t)+t(1-t)\,\,X_3(t)\,\,\omega'(1-t)+X_4(t)\,\omega(1-t)\geq0
\end{align}
where
{\small{
\begin{align*}
X_3(t):=\log\dfrac{1}{t}\left[(1-t)\left(\dfrac{1}{\alpha}+\delta-2B-1\right)+2(C-A-B)t\right]
-(1+2\xi)(1-t)\left[1-\dfrac{1}{p(t)}\right]
\end{align*}}}
and
\begin{align*}
X_4(t):=&\log\dfrac{1}{t}\left[(1-t)^2\left[(\delta-1)\left(\frac{1}{\alpha}-1\right)+(1-B)\left(\frac{1}{\alpha}+\delta-B-1\right)\right]\right.\\
&\left.+(C-A-B)t\times\left[(1-t)\left(\dfrac{1}{\alpha}+\delta-2B-1\right)+(C-A-B-1)t\right]\right]\\
&+(1+2\xi)(1-t)\left(\dfrac{}{}(B-\delta)(1-t)
-(C-A-B)t\right)\left[1-\dfrac{1}{p(t)}\right].
\end{align*}
As, the functions $\omega(1-t)$, $\omega'(1-t)$ and
$\omega''(1-t)$ are non-negative for all values of $t\in(0,1)$,
therefore to prove inequality
\eqref{eq-Gener:convex-MainCond-Gener:hyper-gamma0}, it is
enough to show
\begin{align*}
X_3(t)\geq0\quad{\rm and }\quad X_4(t)\geq0.
\end{align*}
Now, we divide the proof into two cases:

\noindent{\bf{Case (i)}} Let $0<B\leq \delta$. Since the function
    $p(t)$ defined before is non-negative, therefore by a
    small adjustment, the inequality $X_3(t)\geq0$ is
    valid, if
\begin{align*}
\log\dfrac{1}{t}\left[(1-t)\left(\dfrac{1}{\alpha}+\delta-2B-1\right)+2(C-A-B)t\right]
\geq(1+2\xi)(1-t),
\end{align*}
where the parameter $\xi$ is defined above. It is easy to
see that the right side of the above inequality is
positive, hence applying the condition
\eqref{eq-general:cond-log-1-t}, the inequality is true
when
\begin{align}\label{eq-Gener:convex-Main:ineq-gamma0(11)}
(1-t)\left[2\left(\dfrac{1}{\alpha}+\delta-2B-\xi\right)-3\right]+2t\left[\dfrac{}{}2(C-A-B-\xi)-1\right]\geq0.
\end{align}
By the assumptions $B\leq\delta$ and $(C-A-B)>0$, the
condition $X_4\geq0$, holds good if
\begin{align*}
&\log\dfrac{1}{t}\left((1-t)^2\left[(\delta-1)\left(\frac{1}{\alpha}-1\right)+(1-B)\left(\frac{1}{\alpha}+\delta-B-1\right)\right]\right.\\
&\left.+(C-A-B)t\times\left[(1-t)\left(\dfrac{1}{\alpha}+\delta-2B-1\right)+(C-A-B-1)t\right]\right)\\
\geq&(1+2\xi)(1-t)\left(\dfrac{}{}(\delta-B)(1-t)+(C-A-B)t\right).
\end{align*}
For $t\in(0,1)$, the right side term of the above
inequality is positive, hence in view of the condition
\eqref{eq-general:cond-log-1-t}, the above inequality can
be obtained if
{\small{
\begin{align}\label{eq-Gener:convex-Main:ineq-gamma01}
(1-t)^2\left[2\left(\delta-1\right)\left(\dfrac{1}{\alpha}-1\right)+2(1-B)\left(\dfrac{1}{\alpha}+\delta-B-1\right)-(1+2\xi)(\delta-B)\right]\nonumber\\
+t(1-t)\left[(C-A-B)\left(2\left(\dfrac{1}{\alpha}+\delta-2B-1\right)-(1+2\xi)\right)-2(1+2\xi)(\delta-B)\right]\nonumber\\
+2t^2(C-A-B)(C-A-B-2-2\xi)\geq0.
\end{align}}}
Thus, the condition
\eqref{eq-Gener:convex-Main:ineq-gamma01} is true, if the
coefficients of $t^2$, $t(1-t)$, and $(1-t)^2$ are
positive. Now, it remains to prove the following
inequalities:
{\small{
\begin{align}\label{eq-Gener:convex-Main:ineq-gamma0:11}
2\left(\delta-1\right)\left(\dfrac{1}{\alpha}-1\right)+2(1-B)\left(\dfrac{1}{\alpha}+\delta-B-1\right)-(1+2\xi)(\delta-B)\geq0,
\end{align}
\begin{align}\label{eq-Gener:convex-Main:ineq-gamma0:12}
(C-A-B)\left(2\left(\dfrac{1}{\alpha}+\delta-2B-1\right)-(1+2\xi)\right)-2(1+2\xi)(\delta-B)\geq0,
\end{align}
and
\begin{align}\label{eq-Gener:convex-Main:ineq-gamma0:13}
(C-A-B)(C-A-B-2-2\xi)\geq0
\end{align}}}
where $\xi=1-\delta(1-\zeta)$, for
$\left(1-\frac{1}{\delta}\right)\leq\zeta\leq
\left(1-\frac{1}{2\delta}\right)$ and $\delta\geq1$.

\noindent{\bf{Case (ii)}} $B\geq\delta$. It is easy to note that the condition
\eqref{eq-Gener:convex-Main:ineq-gamma0(11)} is true when
\begin{align*}
4B\leq2\left(\dfrac{1}{\alpha}+\delta-\xi\right)-3,
\end{align*}
which clearly means that $B\leq\delta$. Therefore this case is not valid.

Now, for the case $\gamma=0$ $(\mu=0,\nu=\alpha>0)$ and
$\lambda(t)$ defined in
\eqref{eq-Generalized-hypergeometric-fn}, the following result
is stated as under.
\begin{theorem}\label{Thm-Gener:Convex-Generalized-hypergeometric-fn:gamma0}
Let $A,B,C>0$, $1/2\leq\alpha\leq1$ and
$1-\frac{1}{\delta}\leq\zeta\leq1-\frac{1}{2\delta}$, for
$\delta\geq3$. Let $\beta<1$ satisfy
\begin{align*}
\dfrac{\beta-\frac{1}{2}}{1-\beta}=-K\int_0^1 t^{B-1}(1-t)^{C-A-B}\omega(1-t)
q_{0,\alpha}^\delta(t) dt,
\end{align*}
where $q_{0,\alpha}^\delta(t)$  is defined by the differential
equation \eqref{eq-generalized:convex-q:gamma0}, the constant
$K$ and the function $\omega(1-t)$ is given in
\eqref{eq-Generalized-hypergeometric-fn}. Then for
$f(z)\in\mathcal{W}_\beta^\delta(\alpha,0)$, the function

\begin{align*}
H_{A,\,B,\,C}^\delta(f)(z)=\left(K\int_0^1 t^{B-1}(1-t)^{C-A-B}\omega(1-t) \left(\frac{f(tz)}{t}\right)^\delta dt\right)^{1/\delta}
\end{align*}
belongs to $\mathcal{C}_\delta(\zeta)$, if

{\small{
\begin{align*}
C\geq A+B+3\quad{\rm and}\quad
B\leq\min\left\{\dfrac{1}{2}\left(\dfrac{1}{\alpha}+\delta-2\right)\,,\,
\dfrac{\delta\left(\frac{1}{\alpha}-1\right)}{\left(\frac{1}{\alpha}+\delta-1\right)}\,,\,
\dfrac{1}{4}\left(\dfrac{3}{\alpha}+\delta-6\right)\right\}.
\end{align*}}}
\end{theorem}
\begin{proof}
In order to prove the result, it is enough to show the
inequalities \eqref{eq-Gener:convex-Main:ineq-gamma0(11)},
\eqref{eq-Gener:convex-Main:ineq-gamma0:11},
\eqref{eq-Gener:convex-Main:ineq-gamma0:12} and
\eqref{eq-Gener:convex-Main:ineq-gamma0:13} by using the above
hypothesis. The inequality
\eqref{eq-Gener:convex-Main:ineq-gamma0(11)} is valid if
$(C-A-B)\geq1$ and $2B\leq(1/\alpha+\delta-2)$, and
\eqref{eq-Gener:convex-Main:ineq-gamma0:11} is true when
$\delta\left({1}/{\alpha}-1\right)\geq\left({1}/{\alpha}+\delta-1\right)B$.
Since the parameters $(C-A-B)\geq3$ and conditions on $B$
holds, which directly implies that these two inequalities along
with the condition \eqref{eq-Gener:convex-Main:ineq-gamma0:13}
are true.

Further, to prove inequality
\eqref{eq-Gener:convex-Main:ineq-gamma0:12}, it is sufficient
to get the condition

\begin{align*}
(C-A-B)\left(2\left(\dfrac{1}{\alpha}+\delta-2B-1\right)-(1+2\xi)\right)\geq2(1+2\xi)(\delta-B),
\end{align*}
By simple computation, the above expression is holds, if

\begin{align*}
(C-A-B-3)\left(\dfrac{1}{\alpha}+\delta-2B-2\right)+\left(\dfrac{3}{\alpha}+\delta-4B-6\right)\geq0,
\end{align*}
which is clearly true. Hence by the given hypothesis and
Theorem \ref{Thm:Main-Gener:Convex-Decreas} the result directly
follows.
\end{proof}
Let
\begin{align*}
\lambda(t)=\dfrac{\Gamma(c)}{\Gamma(a)\Gamma(b)\Gamma(c-a-b+1)}
t^{b-1}(1-t)^{c-a-b}{\,}_{2}F_1\left(\!\!\!\!
\begin{array}{cll}&\displaystyle c-a,\quad 1-a
\\
&\displaystyle c-a-b+1
\end{array};1-t\right),
\end{align*}
then the integral operator \eqref{eq-weighted-integralOperator}
defined by the above weight function $\lambda(t)$ is the known
as generalized Hohlov operator denoted by
$\mathcal{H}_{a,\,b,\,c}^\delta$. This integral operator was
considered in the work of A. Ebadian \cite{Aghalary} (see also \cite{DeviGenlStar}). When
$\delta=1$, the reduced integral transform was introduced by Y.
C. Kim and F. Ronning \cite{KimRonning} and studied by several
authors later. The operator $\mathcal{H}_{a,\,b,\,c}^\delta$, with
$a=1$ is the generalized Carlson-Shaffer operator
($\mathcal{L}_{b,\,c}^\delta$) \cite{CarlsonShaffer}.

Using the above operators the following results are obtained.
\begin{theorem}\label{Thm-Generalized-Convex-Hohlov}

Let $a,b,c>0$, $\gamma\geq0$ $(\mu\geq,\nu\geq0)$ and
$1-\frac{1}{\delta}\leq\zeta\leq1-\frac{1}{2\delta}$. Let
$\beta\!<\!1$ satisfy

{\small{
\begin{align}\label{eq-Generalized-convex-beta-Hohlov}
\dfrac{\beta}{1\!-\!\beta}=-\dfrac{\Gamma(c)}{\Gamma(a)\Gamma(b)\Gamma(c\!-\!a\!-\!b\!+\!1)}\int_0^1\!\! t^{b-1}(1\!-\!t)^{c-a-b}
{\,}_{2}F_1\left(\!\!\!\!\!\!
\begin{array}{cll}&\displaystyle c\!-\!a,{\,} 1\!-\!a
\\
&\displaystyle c\!-\!a\!-\!b\!+\!1
\end{array};1\!-\!t\right)
q_{\mu,\nu}^\delta(t)dt,
\end{align}}}
where $q_{\mu,\nu}^\delta(t)$ is defined by the differential
equation \eqref{eq-Gener:Convex-q:gamma>0} for $\gamma>0$, and
\eqref{eq-generalized:convex-q:gamma0} for $\gamma=0$. Then for
$f(z)\in\mathcal{W}_\beta^\delta(\alpha,\gamma)$, the function
$\mathcal{H}_{a,\,b,\,c}^\delta(f)(z)$ belongs to the class
$\mathcal{C}_\delta(\zeta)$, whenever
\begin{enumerate}[{\rm(i)}]
  \item
\begin{align*}
&b\leq\min\left\{\dfrac{1}{4}\left(\dfrac{1}{\mu}-3+\delta(3-2\zeta)\right)\,,\,
\dfrac{2}{\left(\delta+1/\mu\right)}\left(\dfrac{(2\delta-1)}{\mu}-\delta+1\right)\right\}\quad {\rm and}\\
&c\geq a+b+2\quad{\rm for}\quad \gamma>0\,\, (1/2\leq\mu\leq1\leq\nu)\quad{\rm and}\quad1\leq\delta\leq2,
\end{align*}
  \item
\begin{align*}
&b\leq\min\left\{\dfrac{1}{2}\left(\dfrac{1}{\alpha}+\delta-2\right)\,,\,
\dfrac{\delta\left(\frac{1}{\alpha}-1\right)}{\left(\frac{1}{\alpha}+\delta-1\right)}\,,\,
\dfrac{1}{4}\left(\dfrac{3}{\alpha}+\delta-6\right)\right\}\quad{\rm and}\\
&c\geq a+b+3\quad{\rm for}\quad 1/2\leq\alpha\leq1,\,\,\gamma=0\,\,
{\rm and}\,\,\delta\geq3.
\end{align*}
\end{enumerate}
\end{theorem}
\begin{proof}
Choosing

\begin{align*}
K=\dfrac{\Gamma(c)}{\Gamma(a)\Gamma(b)\Gamma(c\!-\!a\!-\!b\!+\!1)}\quad{\rm and}\quad
\omega(1-t)=
{\,}_{2}F_1\left(\!\!\!\!\!
\begin{array}{cll}&\displaystyle c-a,{\,} 1-a
\\
&\displaystyle c-a-b+1
\end{array};1-t\right),
\end{align*}
in Theorem
\ref{Thm-Gener:Convex-Generalized-hypergeometric-fn:gamma>0}
and \ref{Thm-Gener:Convex-Generalized-hypergeometric-fn:gamma0}
for the case $\gamma>0$ and $\gamma=0$, respectively to get the
required result.
\end{proof}
For $a=1$, Theorem \ref{Thm-Generalized-Convex-Hohlov} lead to
the following particular cases which are of independent interest.
\begin{corollary}
Let $b, c>0$, $\gamma\geq0$ $(\mu\geq0, \nu\geq0)$ and
$1-\frac{1}{\delta}\leq\zeta\leq1-\frac{1}{2\delta}$. Let
$\beta<1$ satisfy

\begin{align*}
\dfrac{\beta}{(1-\beta)}=-\dfrac{\Gamma(c)}{\Gamma(b)\Gamma(c-b)}\int_0^1
t^{b-1}(1-t)^{c-b-1}q_{\mu,\nu}^\delta(t)dt,
\end{align*}
where $q_{\mu,\nu}^\delta(t)$ is defined by the differential
equation \eqref{eq-Gener:Convex-q:gamma>0} for $\gamma>0$, and
\eqref{eq-generalized:convex-q:gamma0} for $\gamma=0$. Then for
$f(z)\in\mathcal{W}_\beta^\delta(\alpha,\gamma)$, the function
$\mathcal{L}_{b,\,c}^\delta(f)(z)$ belongs to the class
$\mathcal{C}_\delta(\zeta)$, whenever
\begin{enumerate}[{\rm(i)}]
  \item
\begin{align*}
&b\leq\min\left\{\dfrac{1}{4}\left(\dfrac{1}{\mu}-3+\delta(3-2\zeta)\right)\,,\,
\dfrac{2}{\left(\delta+1/\mu\right)}\left(\dfrac{(2\delta-1)}{\mu}-\delta+1\right)\right\}\quad{\rm and}\\
&c\geq b+3\quad {\rm for}\,\,\, \gamma>0\, (1/2\leq\mu\leq1\leq\nu)\,\, and\,\, 1\leq\delta\leq2
\end{align*}
  \item
\begin{align*}
&b\leq\min\left\{\dfrac{1}{2}\left(\dfrac{1}{\alpha}+\delta-2\right)\,,\,
\dfrac{\delta\left(\frac{1}{\alpha}-1\right)}{\left(\frac{1}{\alpha}+\delta-1\right)}\,,\,
\dfrac{1}{4}\left(\dfrac{3}{\alpha}+\delta-6\right)\right\}\quad{\rm and}\\
&c\geq b+4\quad {\rm for}\,\,\,1/2\leq\alpha\leq1,\,\,\gamma=0\,\, {\rm and}\,\, \delta\geq3
\end{align*}
\end{enumerate}
\end{corollary}
\begin{corollary}
Let $b,c>0$, $\gamma\geq0$ $(\mu\geq0,\nu\geq0)$ and
$1-\frac{1}{\delta}\leq\zeta\leq1-\frac{1}{2\delta}$. Let
$\beta_0<\beta<1$, where
\begin{align*}
\beta_0=1-\dfrac{1}{\left(1-{\,}
{\,}_6F_5
\left(
\begin{array}{cll}&\displaystyle \quad\quad 1,b,(1+\delta),(2-\xi),\dfrac{\delta}{\mu},\dfrac{\delta}{\nu},
\\
&\displaystyle c,\delta,(1-\xi),\left(1+\dfrac{\delta}{\mu}\right),\left(1+\dfrac{\delta}{\nu}\right)
\end{array}{\,};{\,}-1\right)
\right)}.
\end{align*}
Then, for $f\in\mathcal{W}_{\beta}^\delta(\alpha,\gamma)$, the
function
$\mathcal{L}_{b,c}^\delta(f)(z)\in\mathcal{C}_\delta(\zeta)$,
whenever
\begin{enumerate}[{\rm(i)}]
  \item
\begin{align*}
&b\leq\min\left\{\dfrac{1}{4}\left(\dfrac{1}{\mu}-3+\delta(3-2\zeta)\right)\,,\,
\dfrac{2}{\left(\delta+1/\mu\right)}\left(\dfrac{(2\delta-1)}{\mu}-\delta+1\right)\right\}\quad{\rm and}\\
&c\geq b+3\quad {\rm for}\,\,\, \gamma>0\, (1/2\leq\mu\leq1\leq\nu)\,\, and\,\, 1\leq\delta\leq2
\end{align*}
  \item
\begin{align*}
&b\leq\min\left\{\dfrac{1}{2}\left(\dfrac{1}{\alpha}+\delta-2\right)\,,\,
\dfrac{\delta\left(\frac{1}{\alpha}-1\right)}{\left(\frac{1}{\alpha}+\delta-1\right)}\,,\,
\dfrac{1}{4}\left(\dfrac{3}{\alpha}+\delta-6\right)\right\}\quad{\rm and}\\
&c\geq b+4\quad {\rm for}\,\,\,1/2\leq\alpha\leq1,\,\,\gamma=0\,\, {\rm and}\,\, \delta\geq3
\end{align*}
\end{enumerate}
\end{corollary}
\begin{proof}
Putting $a=1$ in \eqref{eq-Generalized-convex-beta-Hohlov} and
on further using \eqref{eq-gener:convex-q:hyper:series} will
give
{\small{
\begin{align*}
\dfrac{\beta}{1\!-\!\beta}=-\dfrac{\Gamma(c)}{\Gamma{(b)}\Gamma{(c\!-\!b)}}\int_0^1\!\!t^{b-1}(1\!-\!t)^{c-b-1}
{\,}_5F_4\!
\left(\!\!\!\!\!\!\!\!
\begin{array}{cll}&\displaystyle \quad\quad 1,(1+\delta),(2-\xi),\dfrac{\delta}{\mu},\dfrac{\delta}{\nu}
\\
&\displaystyle \delta,(1\!-\!\xi),\left(1\!+\!\dfrac{\delta}{\mu}\right),\left(1\!+\!\dfrac{\delta}{\nu}\right)
\end{array}{\,};{\,}-t\right)dt
\end{align*}}}
or equivalently,

{\small{
\begin{align*}
\dfrac{\beta}{1-\beta}=-\dfrac{\Gamma(c)}{\Gamma{(b)}\Gamma{(c\!-\!b)}}\int_0^1\!\!t^{b-1}(1\!-\!t)^{c-b-1}
\left(\sum_{n=0}^{\infty}\dfrac{(1+\delta)_n(2-\xi)_n\left(\dfrac{\delta}{\mu}\right)_n\left(\dfrac{\delta}{\nu}\right)_n
(-1)^n}{(\delta)_n(1-\xi)_n\left(n+\dfrac{\delta}{\nu}\right)_n\left(n+\dfrac{\delta}{\mu}\right)_n} t^n\right)dt.
\end{align*}}}
%
Now a simple computation leads to  
\begin{align*}
\dfrac{\beta}{1-\beta}
=-{\,}_6F_5\!
\left(\begin{array}{cll}&\displaystyle \quad\quad 1,b,(1+\delta),(2-\xi),\dfrac{\delta}{\mu},\dfrac{\delta}{\nu}
\\
&\displaystyle c,\delta,(1-\xi),\left(1+\dfrac{\delta}{\mu}\right),\left(1+\dfrac{\delta}{\nu}\right)
\end{array}{\,};{\,}-1\right).
\end{align*}
Thus, applying Theorem \ref{Thm-Generalized-Convex-Hohlov} will
give the required result.
\end{proof}
Consider
\begin{align}\label{eq-komatu_operator}
\lambda(t)=\dfrac{(1+k)^p}{\Gamma(p)}t^{k}\left(\log\dfrac{1}{t}\right)^{p-1},
\quad \delta\geq 0\quad k>-1.
\end{align}
Then the integral operator \eqref{eq-weighted-integralOperator}
defined by the above weight function $\lambda(t)$ is the known
as generalized Komatu operator denoted by $(F_{k,\,p}^\delta)$.
This integral operator was considered in the work of A. Ebadian
\cite{Aghalary}. When $\delta=1$, the operator is reduced to
the one introduced by Y. Komatu \cite{komatu}.

Now, we state the following result.
\begin{theorem}
Let $\gamma\geq0$ $(\mu\geq,\nu\geq0)$, $k>-1$, $p\geq1$ and
$1-\frac{1}{\delta}\leq\zeta\leq1-\frac{1}{2\delta}$. Let
$\beta\!<\!1$ satisfy \eqref{Beta-Cond-Generalized:Convex},
where $\lambda(t)$ is given in \eqref{eq-komatu_operator}. Then
for $f(z)\in\mathcal{W}_\beta^\delta(\alpha,\gamma)$, the
function $F_{k,p}^\delta(f)(z)\in\mathcal{C}_\delta(\zeta)$,
whenever

\begin{enumerate}[{\rm(i)}]
  \item {\small{
\begin{align*}
&-1<k\leq\min\left\{\dfrac{1}{4}\left(\dfrac{1}{\mu}-3+\delta(3-2\zeta)\right)-1\,,\,
\dfrac{2}{\left(\delta+1/\mu\right)}\left(\dfrac{(2\delta-1)}{\mu}-\delta+1\right)-1\right\}\\
&{\rm and}\quad p\geq 1\,\, {\rm for}\,\, \gamma>0\,(1/2\leq\mu\leq1\leq\nu)\,\, {\rm and}\,\,
1\leq\delta\leq2,
\end{align*}}}
  \item {\small{
\begin{align*}
&-1<k\leq\min\left\{\dfrac{1}{2}\left(\dfrac{1}{\alpha}+\delta-4\right)\,,\,
\dfrac{\delta\left(\frac{1}{\alpha}-1\right)}{\left(\frac{1}{\alpha}+\delta-1\right)}-1\,,\,
\dfrac{1}{4}\left(\dfrac{3}{\alpha}+\delta-10\right)\right\}\\
&{\rm and}\quad p\geq 2\,\, {\rm for}\,\, 1/2\leq\alpha\leq1,\,\, \gamma=0\,\,{\rm  and}\,\,\delta\geq3.
\end{align*}}}
\end{enumerate}
\end{theorem}
\begin{proof}
Letting $(C-A-B)=p-1$, $B=k+1$ and
$\omega(1-t)=\left(\frac{\log(1/t)}{(1-t)}\right)^{p-1}$.
Therefore $\lambda(t)$ given in
\eqref{eq-Generalized-hypergeometric-fn} can be represented as

\begin{align*}
\lambda(t)=Kt^{k}(1-t)^{p-1}\omega(1-t),\quad{\mbox{where}\quad}  K=\dfrac{(1+k)^p}{\Gamma(p)}.
\end{align*}
Now, by the given hypothesis the result directly follows from
Theorem
\ref{Thm-Gener:Convex-Generalized-hypergeometric-fn:gamma>0}
and \ref{Thm-Gener:Convex-Generalized-hypergeometric-fn:gamma0}
for the case $\gamma>0$ and $\gamma=0$, respectively.
\end{proof}


\end{document}